\documentclass[11pt, oneside]{amsart}   	
\usepackage[bottom=47mm]{geometry}                		
\usepackage{graphicx}				
\usepackage{amssymb}
\usepackage{xcolor}
\usepackage{pgfplots,tikz}
\usetikzlibrary{patterns, shapes.geometric, intersections, calc}
\definecolor{imayou}{RGB}{154, 154, 235}
\definecolor{usuai}{RGB}{9, 150, 126}
\definecolor{persred}{RGB}{154,63,63}
\definecolor{sand}{RGB}{201, 177, 60}
\pgfplotsset{compat=1.18}
\usepackage{hyperref}
\usepackage{amsmath}
\usepackage{amsthm}
\usepackage{mathrsfs}
\usepackage{mathtools}
\usepackage[title]{appendix}
\mathtoolsset{showonlyrefs=true}
\usepackage{xcolor}
\usepackage{bbm}
\usepackage{enumerate, comment} 
\usepackage{ulem}

\title[Geometry of wave damping on the torus]{Geometry of wave damping on the torus}

\author{Kiril Datchev, Perry Kleinhenz, and Antoine Prouff}

\date{}							

\theoremstyle{definition}

\theoremstyle{theorem}
\newtheorem{theorem}{Theorem}[section]
\newtheorem{lemma}[theorem]{Lemma}
\newtheorem{corollary}[theorem]{Corollary}
\newtheorem{proposition}[theorem]{Proposition}
\theoremstyle{definition}
\newtheorem{definition}[theorem]{Definition}
\newtheorem{remark}[theorem]{Remark}
\newtheorem{example}[theorem]{Example}
\numberwithin{equation}{section}

\newcommand{\Rb}{\mathbb{R}}
\newcommand{\Zb}{\mathbb{Z}}

\newcommand{\Dc}{\mathcal{D}}

\newcommand{\Tb}{\mathbb{T}}

\newcommand{\ra}{\rightarrow}

\renewcommand{\>}{\right\rangle}
\newcommand{\e}{\varepsilon}
\renewcommand{\d}{\delta}

\newcommand{\nm}[1]{\left| \left| #1 \right| \right|}

\newcommand{\lp}[2]{ \nm{#1}_{L^{#2}}}

\newcommand{\hp}[2]{\nm{#1}_{H^{#2}}}

\newcommand{\ltwo}[1]{\lp{#1}{2}}

\newcommand{\p}{\partial}

\newcommand{\supp}{\text{supp }}
\newcommand{\T}{\mathbb{T}}

\newcommand{\Lc}{\mathcal{L}}

\newcommand{\ti}{\widetilde}

\newcommand{\Sb}{\mathbb{S}}

\renewcommand{\Re}{\text{Re }}

\newcommand{\pl}{\rho(\lambda)}

\newcommand{\dist}{\text{dist}}

\newcommand{\Gc}{\mathcal{G}}

\DeclareMathOperator{\Span}{span}

\begin{document}
    \begin{abstract}
        Energy decay rates of damped waves on the torus 
        depend on the behavior of the damping near the undamped region and on the geometry of the damped set. In this paper we refine these geometric considerations, by introducing the concept of 
        order of a glancing undamped point, 
        and estimating decay rates in terms of this order. 
        The proof is based on generalizing an averaging argument due to Sun.
        We also show that damping sets which attain these improvements are generic among 
        polygons and  
        smooth curves. 
    \end{abstract}
	\maketitle
\section{Introduction}

We explore how  decay rates for the damped wave equation on $\T^2 :=\Rb^2/\Zb^2$ depend on the shape of the damping set $\omega$ near points where undamped geodesics touch $\p \omega$. We first consider three examples in Figure \ref{f:1ex}.

    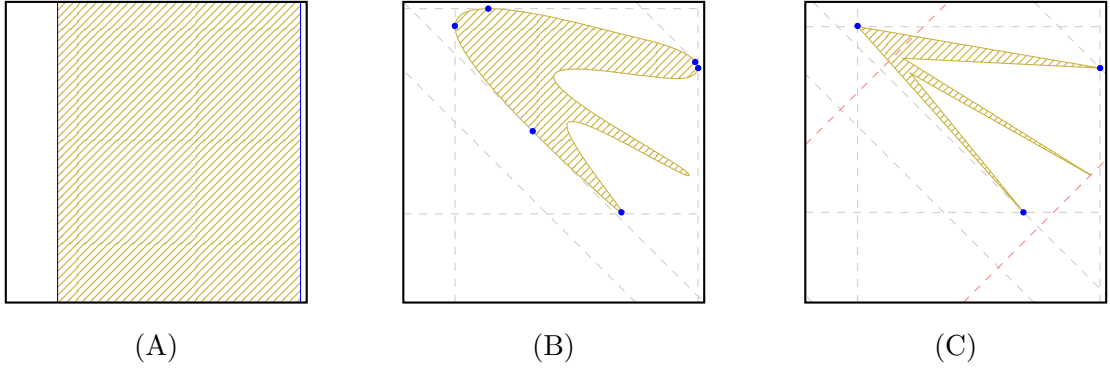
\begin{figure}[ht]
    \begin{tikzpicture}
        \draw[sand, pattern=north east lines, pattern color=sand] 
        (0.688,0) -- (3.918,0) -- (3.918,4) -- (0.688,4) -- cycle;
        \draw[blue] (.688,0) -- (0.688,4);\draw[blue] (3.918,0) -- (3.918,4);
       \draw[thick] (0,0) rectangle (4,4);\node at (2,-.6) {(A)};
    \end{tikzpicture}
        \hspace{1cm}     \begin{tikzpicture} 
        \draw[sand, pattern=north east lines, pattern color=sand]
        plot [smooth cycle, tension=0.8] coordinates {
          (2.86,1.22) (2.2,2.4) (3.8,1.7) (2,3) (3.8,3) (3.2,3.5) (.69,3.7)  
        }; 
    \draw[dashed, gray!40] (0,4) -- (4,0);  \filldraw[blue]  (1.72,2.28) circle (1pt);
    \draw[dashed, gray!40] (3.09,4) -- (4,3.09);\draw[dashed, gray!40] (3.09,0) -- (0,3.09); 
    \draw[dashed, gray!40] (0,1.18) -- (4,1.18); \filldraw[blue]  (2.9,1.2) circle (1pt);
    \draw[dashed, gray!40] (0,3.91) -- (4,3.91);  \filldraw[blue]  (1.13,3.91) circle (1pt);
    \draw[dashed, gray!40] (.688,0) -- (.688,4); \filldraw[blue]  (.685,3.68) circle (1pt);
    \draw[dashed, gray!40] (3.918,0) -- (3.918,4); \filldraw[blue]  (3.88,3.2) circle (1pt); \filldraw[blue]  (3.92,3.12) circle (1pt);
           \draw[thick] (0,0) rectangle (4,4); 
        \node at (2,-.6) {(B)};
    \end{tikzpicture}
\hspace{1cm}
    \begin{tikzpicture} 
  \draw[sand, pattern=north east lines, pattern color=sand] (2.92,1.18) -- (1.4,3.05) -- (3.795,1.695) -- (1.3,3.25)-- (3.917,3.12) --  (.7,3.67) -- cycle;
    \draw[dashed, gray!40] (0.1,4) -- (4,0.1);\draw[dashed, gray!40] (0.1,0) -- (0,0.1); \draw[dashed, gray!40] (0,1.2) -- (4,1.2);\filldraw[blue]  (2.9,1.2) circle (1pt);
    \draw[dashed, gray!40] (3.918,0) -- (3.918,4);\draw[dashed, gray!40] (3.04,4) -- (4,3.04); \draw[dashed, gray!40] (3.04,0) -- (0,3.04);  \filldraw[blue]  (3.92,3.12) circle (1pt);
     \draw[dashed, gray!40] (.695,0) -- (.695,4); \draw[dashed, gray!40] (0,3.67) -- (4,3.67); \filldraw[blue]  (.695,3.68) circle (1pt); 
    \draw[dashed, red!50] (0,2.1) -- (1.9,4); \draw[dashed, red!50] (2.1,0) -- (4,1.9);
       \draw[thick] (0,0) rectangle (4,4);
        \node at (2,-.6) {(C)};
   \end{tikzpicture}
        \hspace{1cm}

    \caption{Three examples of damping sets $\omega$ on the torus, bounded by (A) a cylinder, (B) a curve which at the six blue points has nonzero curvature, and (C)  a polygon.  Blue indicates points  where undamped geodesics (shown in gray) touch $\partial \omega$; in (C) the red geodesic intersects $\omega$ so it is not undamped. }\label{f:1ex}
    \end{figure}
   
    Our first result says that if the damping near the blue points is approximately a power of the distance to the undamped set, then we have a quantitative improvement in the decay rate for cases (B) and (C) relative to case (A).
    \begin{theorem}\label{t:1ex}
    Consider the damped wave equation
    \begin{equation} \label{e:dampwe}
		\begin{cases}
			(\p_t^2 - \Delta + W \p_t) u =0, \quad (z,t) \in \T^2\times\Rb, \\
			(u, \p_t u)|_{t=0} = (u_0, u_1) \in H^2(\T^2) \times H^1(\T^2),
		\end{cases}
	\end{equation}
    where $W \in W^{9, \infty}(\T^2)$ is nonnegative, not identically zero, and $|\partial^\gamma W| \lesssim W^{1-\frac{|\gamma|}{4}}$ for $|\gamma| \le 2$. 
    
    Suppose that the damping set 
    \[\omega :=\{z \in \mathbb T^2 \colon W(z)>0\}\]
    is as in Figure \ref{f:1ex}, and let $d(z)=\dist(z, \T^2 \backslash \omega)$ be the distance from $z$ to the undamped region. Assume there exists $\beta\ge9$ such that,  in a neighborhood of the blue points,
			\begin{equation}
				  d(z)^{\beta} \lesssim W(z) \lesssim d(z)^{\beta}.
			\end{equation}
    Then we have polynomial energy decay:
    \begin{equation}\label{eq:stable}
    E(u,t):= \frac 12 \int_{\mathbb T^2} |\nabla u(z,t) |^2 + |\p_t u(z,t) |^2 dz \lesssim t^{-2\alpha} \left( \hp{u_0}{2}^2 + \hp{u_1}{1}^2 \right),
    \end{equation}
    where
    \begin{equation}
    \alpha =  1 - \frac{1}{\beta +3}, \qquad\alpha = 1 - \frac{1}{\beta + \frac 12 +3} , \qquad \text{and} \qquad   \alpha = 1 -  \frac{1}{\beta + 1 +3},
    \end{equation}
    in cases (A), (B), and (C) respectively.
    \end{theorem}
    \begin{remark}
        Theorem~\ref{t:1ex} gives a bona fide improvement for cases (B) and (C) over case (A), since decay at rate $\alpha = 1 - \frac{1}{\beta + 3}$ is known to be sharp for damping coefficients invariant in one direction ~\cite{Kleinhenz2019}.
        Note that $\omega_A \supset \omega_B \supset \omega_C$, but the energy decay rates improve as we go from case (A) to (B) to (C). This is exactly due to the differing shapes of $\p \omega$ near trajectories that intersect $\p \omega$ but not $\omega$.
    \end{remark}
    
    Case (A) of Theorem \ref{t:1ex} is exactly \cite{DatchevKleinhenz2020}. Cases (B) and (C) of Theorem \ref{t:1ex} are consequences of the more general Theorem \ref{t:uniform} below. In the language of Theorem \ref{t:uniform}, the glancing set ${\mathcal G}$ is the set of blue points in Figure \ref{f:1ex}, and these points have order $2$ in case~(B) and order $1$ in case (C). The glancing  (i.e.\ undamped) lines through each blue point are drawn in gray; for example, of the four acute angles in Figure \ref{f:1ex}(C), three have two glancing lines each. But the last acute angle has no glancing lines; any line which is locally glancing, like the red one in the figure, eventually enters the damping set $\omega$. 

\subsection{Background}


Polynomial decay rates have been studied for damping on the square in \cite{LiuRao2005}, on partially rectangular domains (including tori) in \cite{BurqHitrik2007}, and for a degenerately hyperbolic undamped set in \cite{csvw}.

In the setting of \eqref{e:dampwe}, whenever $\omega$ is nonempty, \eqref{eq:stable} holds with $\alpha = 1/2$: see  \cite[Theorem 2.3]{AL14} and~\cite{Macia2010,BurqZworski2012,BurqZworski2019}. 
On the other hand, if some geodesics are undamped, i.e.\ do not intersect $\omega$, then \eqref{eq:stable} does not hold for any $\alpha > 1$  \cite[Theorem~2.5]{AL14}.

	
	
	Making additional assumptions on $W$ refines this rate. If $W(x,y)=(|x|-\sigma)_+^{\beta}$, $\beta>-1$, near $\{x = \sigma\}$, then \eqref{eq:stable} holds with $\alpha=1 - \frac{1}{\beta+3}$, and there are solutions decaying no faster than this rate \cite{Kleinhenz2019, DatchevKleinhenz2020, KleinhenzWang2022}. 
    That is, for $y$-invariant damping supported on a strip the polynomial growth of the damping near $\partial \omega$ determines the sharp polynomial energy decay rate of solutions. 
    
    When $W=d(z)^{\beta}$, and $\omega$ is a locally strictly convex set with positive curvature, the sharp energy decay rate is faster. Roughly, $\beta$ is replaced by $\beta+\frac{1}{2}$, so \eqref{eq:stable} holds with $\alpha=1 - \frac{1}{\beta+\frac{1}{2}+3}$, and there are solutions decaying no faster than this rate \cite{Sun23}. Recently, this improvement to the energy decay rate was generalized to damping growing polynomial-logarithmically and extended to $\omega$ equal to a super-ellipse $\{ |\frac{x}{a}|^m + |\frac{y}{b}|^n < 1 \}$, which has zero curvature at some points \cite{Kleinhenz2025}.

	In this paper, we further relax and generalize the geometric assumptions on $\omega$, and sometimes obtain an even stronger improvement over the $y$-invariant decay rate.

\subsection{Definitions, notation and main results} 
    Let $W \in C^0(\T^2)$, $W \ge 0$, $W \not \equiv 0$. 
    
	For $z\in \T^2$ and $v \in \Sb^1$, the line starting from $z$ in the direction $v$ is
    \begin{equation*} \label{e:def-geodesic}
    L=L_{z,v}  := \{z+tv \in \T^2 \colon t \in \Rb\}.
    \end{equation*}

\begin{definition}[Glancing lines] \label{def:glancing-line}
    A line $L = L_{z,v} $ is \textit{glancing} if $L  \cap \omega = \varnothing$ and $L  \cap \partial \omega \ne \varnothing$. Such a line is:
    \begin{enumerate}
        \item a \textit{one-sided glancing line} if there are $\varepsilon_0>0$ and $w \in \mathbb S^1$ such that $v \cdot w =0$ and $L_{z+s w,v}\cap \omega = \varnothing$ for $s \in (0,\varepsilon_0)$;
        \item a \textit{two-sided glancing line} otherwise.
    \end{enumerate}
    We denote by $\mathcal{L}_1(v)$ and $\mathcal{L}_2(v)$ the sets of {one-sided} and {two-sided} glancing lines of direction $v$ respectively, and we define
    \begin{equation*}
    \mathcal{L}_1 := \bigcup_{v \in \Sb^1} \mathcal{L}_1(v)
        \qquad \textrm{and} \qquad
    \mathcal{L}_2 := \bigcup_{v \in \Sb^1} \mathcal{L}_2(v) .
    \end{equation*}
    
    The set of \textit{glancing directions} is 
    \[\mathcal V = \{v \in \mathbb S^1 \colon \text{there exists } p \in \mathbb T^2 \text{ such that } L_{p,v}  \text{ is glancing}\}.\]

\end{definition}


We say that a direction $v \in \Sb^1$ is \textit{rational} if $cv \in \Zb^2$ for some $c > 0$, and irrational otherwise. Recall that geodesics with irrational direction are dense in $\Tb^2$. 
In particular, all glancing directions are rational. Furthermore, as the below lemma shows, there are only finitely many glancing directions. 
    
\begin{lemma} \label{l:finite-trapped}
    Assume $\omega$ contains an open ball of radius $\varepsilon > 0$. Then there exist at most $1/\varepsilon^2$ (rational) directions $v \in \Sb^1$ such that $L_{z, v} \cap \omega = \varnothing$ for some $z \in \Tb^2$.
\end{lemma}
See Appendix~\ref{app} for a proof and Figure \ref{f:disk} for an example.

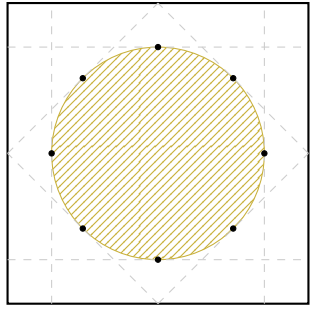
\begin{figure}[ht]
\begin{tikzpicture} 
   \draw[thick] (0,0) rectangle (4,4);
  \draw[sand, pattern=north east lines, pattern color=sand] 
    (2,2) circle (1.414);
    \draw[dashed, gray!40] (2,0) -- (0,2); \filldraw (1,1) circle (1pt); 
    \draw[dashed, gray!40] (4,2) -- (2,4); \filldraw (3,3) circle (1pt); 
    \draw[dashed, gray!40] (0,2) -- (2,4); \filldraw (1,3) circle (1pt); 
    \draw[dashed, gray!40] (2,0) -- (4,2); \filldraw (3,1) circle (1pt);    
    \draw[dashed, gray!40] (0,3.4142) -- (4,3.4142); \filldraw (2,3.4142) circle (1pt); 
    \draw[dashed, gray!40] (3.4142,0) -- (3.4142,4); \filldraw (3.4142,2) circle (1pt); 
    \draw[dashed, gray!40] (0,.5858) -- (4,0.5858);  \filldraw (2,0.5858) circle (1pt); 
    \draw[dashed, gray!40] (.5858,0) -- (0.5858,4);  \filldraw (0.5858,2) circle (1pt); 
\end{tikzpicture}
    \caption{The disk of diameter $1/\sqrt2$ has eight glancing directions (one for each eighth root of unity) and six glancing lines. The horizontal and vertical glancing lines are one-sided, and the diagonal ones are two-sided. If the damping set is a superset of this disk, then its glancing directions are a subset of these eight. If the damping set is a suitable subset of this disk, then the glancing directions are still  the same; see Figure \ref{f:disk2} for  examples.
    }\label{f:disk}
\end{figure}

\begin{definition}[Glancing points] \label{def:glancing-pt}
    Let $L$ be a glancing line with direction $v$ and let $z \in L \cap \partial \omega$. We say that the point $z$ is:
    \begin{enumerate}
        \item a \textit{one-sided glancing point} (relative to the direction $v$) if there exist a neighborhood $U \subset \Tb^2$ of $z$ and $w \in \mathbb S^1$ such that $v \cdot w =0$ and $L_{z+s w,v}\cap U \cap \omega = \varnothing$ for any $s > 0$;
        \item a \textit{two-sided glancing point} (relative to the direction $v$) otherwise.
    \end{enumerate}
\end{definition}

    We use the following notation for sets of glancing points:     
    \[
    {\mathcal G} = \bigcup_{v \in \mathcal V} {\mathcal G}(v), \qquad  {\mathcal G}(v) = \{ z \in \partial \omega \colon z \in L \text{ for some glancing line } L \text{ with direction }v\},
    \]
Similarly we write for the sets of one-sided and two-sided glancing points
    \[
    {\mathcal G}_1 = \bigcup_{v \in \mathcal V} {\mathcal G}_1(v), \qquad {\mathcal G}_2 = \bigcup_{v \in \mathcal V} {\mathcal G}_2(v).
    \]




\begin{remark}
Notice that a {two-sided} glancing line may contain only {one-sided} glancing points; this is the case for the two-sided glancing lines in Figure \ref{f:disk}.
\end{remark}
    
    The next definition makes precise and generalizes the behavior of $\partial \omega$ near the marked points of Figures \ref{f:1ex} and \ref{f:disk}. Figure~\ref{f:d-order} has further examples illustrating the definition.
\begin{definition}\label{def:order}[Order of a glancing point]
Let $L$ be a glancing line, let $z \in \partial \omega \cap L$, and let $\eta >0$. We say that $z$ has \textit{order} $\eta$, if there exists an affine coordinate chart $(\psi,U)$ about $z$ such that
\begin{equation}\label{e:psil}
\psi(z) = (0,0)
    \qquad \textrm{and} \qquad
\psi(L \cap U) = \{(x,y) \in \psi(U) \colon  x=0 \} ,
\end{equation}
and there exist constants $C_{out} > C_{in} >1$ such that the following holds. Defining the sets
\begin{equation*}
\Omega_\eta
    := \{(x,y) \in \psi(U) : |y|^{\eta} \leq C_{out} |x| \} ,
     \qquad
\mathcal{F}_\eta
    := \{(x,y) \in \psi(U) : C_{in}^{-1} |y|^{\eta} \leq |x| \le C_{in} |y|^{\eta} \} ,
\end{equation*}
we have
\begin{itemize}
\item in the case $z \in \Gc_1$,
\begin{equation} \label{e:order-1-sided}
\mathcal{F}_\eta \cap \{x \ge 0, y \ge 0\}
    \subset \psi(\omega \cap U)
    \subset \Omega_\eta \cap \{ x \ge 0 \} ;
\end{equation}
\item in the case $z \in \Gc_2$,
\begin{equation} \label{e:order-2-sided}
\mathcal{F}_\eta \cap \{y \ge 0\}
    \subset \psi(\omega \cap U)
    \subset \Omega_\eta
        \qquad \textrm{or} \qquad
\mathcal{F}_\eta \cap \{xy \ge 0\}
    \subset \psi(\omega \cap U)
    \subset \Omega_\eta .
\end{equation}
\end{itemize}
\end{definition}

\begin{figure}[h]
       \begin{tikzpicture} 
  \draw[sand, pattern=north east lines,  pattern color=sand]plot [domain=1:0] (\x,{1*\x^.5}) --  plot [domain=0:1] (\x,-{1*\x^.5});

  \filldraw (0,0) circle (1pt);

  \draw[gray!90] (-1.3,0) -- (1.3,0);
  \draw[gray!90] (0,-1.3) -- (0,1.3);
   \end{tikzpicture}
\hspace{1cm}       
\begin{tikzpicture} 
  \draw[sand, pattern=north east lines,  pattern color=sand]plot [domain=1:0](\x,{0.5*abs(\x)^.5}) --  plot [domain=0:1] (\x,{1*abs(\x)^.5});

  \filldraw (0,0) circle (1pt);

  \draw[gray!90] (-1.3,0) -- (1.3,0);
  \draw[gray!90] (0,-1.3) -- (0,1.3);
   \end{tikzpicture}
\hspace{1cm}
\begin{tikzpicture} 
  \draw[sand, pattern=north east lines,  pattern color=sand]plot [domain=1:0](\x,{0.5*abs(\x)^.5}) --  plot [domain=0:1] (\x,{1*abs(\x)^.5});

  \draw[sand, pattern=north east lines,  pattern color=sand]plot [domain=-1:0](\x,{0.5*abs(\x)^.5}) --  plot [domain=0:-1] (\x,{1*abs(\x)^.5});;
  
  \filldraw (0,0) circle (1pt);

  \draw[gray!90] (-1.3,0) -- (1.3,0);
  \draw[gray!90] (0,-1.3) -- (0,1.3);
   \end{tikzpicture}
\hspace{1cm}
\begin{tikzpicture} 
  \draw[sand, pattern=north east lines,  pattern color=sand]plot [domain=1:0](\x,{0.5*abs(\x)^.5}) --  plot [domain=0:1] (\x,{1*abs(\x)^.5});

  \draw[sand, pattern=north east lines,  pattern color=sand]plot [domain=-1:0](\x,-{0.5*abs(\x)^.5}) --  plot [domain=0:-1] (\x,-{1*abs(\x)^.5});
  
  \filldraw (0,0) circle (1pt);

  \draw[gray!90] (-1.3,0) -- (1.3,0);
  \draw[gray!90] (0,-1.3) -- (0,1.3);
   \end{tikzpicture}
     \caption{ The first two are one-sided glancing points, obeying \eqref{e:order-1-sided}. The last two are two-sided glancing points, obeying respectively the first and the second of \eqref{e:order-2-sided}.}
     \label{f:d-order}
   \end{figure}
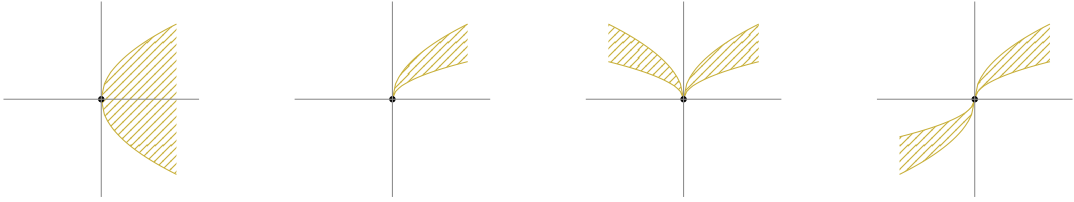

	For example, if $\partial \omega$ is a  curve with nonzero curvature at $z$, as in  Figure~\ref{f:1ex}(B) or Figure~\ref{f:disk},  then $z$ has order $2$.  If $\partial \omega$ is a polygon  with $z$ a vertex, as in Figure \ref{f:1ex}(C), then $z$ has order~$1$.  The same holds for suitable curvilinear polygons, as in Figure \ref{f:order}. Note that a single point can have different orders depending on the glancing direction; see Figure~\ref{f:order-fang}.

    Our methods also apply to infinite order glancing points, but since this case is more cumbersome and exhibits no decay improvement, we omit it for simplicity.

    \begin{remark}\label{r:finiteG}
        If every point of $\mathcal G$ has an order, then $\mathcal G$ is finite. Indeed, by Lemma \ref{l:finite-trapped} it is enough to check that $\mathcal G(v)$ is finite for each $v$. For this observe that, for each $v$, the set of glancing lines with direction $v$ is closed, and hence so is $\mathcal G(v)$. Since Definition \ref{def:order} forbids any point with an order from being an accumulation point of glancing points, it follows that the set of glancing points is finite.
    \end{remark}
    
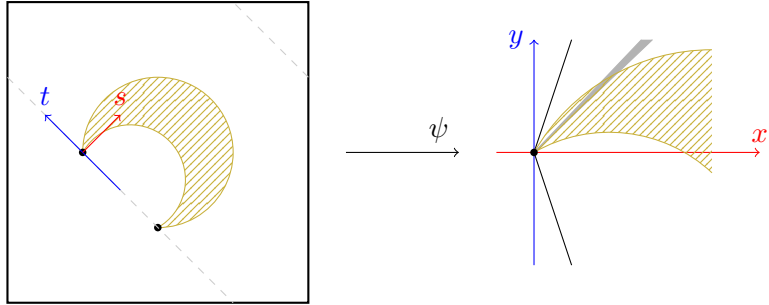
\begin{figure}[h]

\begin{tikzpicture} 
   \draw[thick] (0,0) rectangle (4,4);
      \coordinate (A) at (2,2);
  \coordinate (B) at (1.634,1.634); 
  \begin{scope}[overlay]
  \path [name path=A] (A) circle [radius=1];
  \path [name path=B] (B) circle [radius=.732]; 
  \path [name intersections={of=A and B, by={p1,p2}}];
  \end{scope}
  \draw [sand, pattern=north east lines, pattern color=sand]  let
    \p1=(A),\p2=(B),\p3=(p1),\p4=(p2),
    \n1={veclen(\x3-\x1,\y3-\y1)},
    \n2={atan2(\y3-\y1,\x3-\x1)}, \n3={atan2(\y4-\y1,\x4-\x1)},
    \n4={veclen(\x3-\x2,\y3-\y2)},
    \n5={atan2(\y3-\y2,\x3-\x2)}, \n6={atan2(\y4-\y2,\x4-\x2)} in
    ($(A)+(\n2:\n1)$) arc (\n2:\n3:\n1) arc(\n6:\n5:\n4) -- cycle;
  \foreach \n in {p1,p2}
     \fill (\n) circle [radius=.05];
        \draw[dashed, gray!40] (0,3) -- (3,0); \draw[dashed, gray!40] (4,3) -- (3,4);
    \draw[->] (4.5,2) -- (6,2) node[above left] {$\psi$};
     
   \draw[red, ->] (6.5,2) -- (10,2) node[above] {$x$}; \draw[blue, ->] (7,0.5) -- (7,3.5) node[left] {$y$};
    \draw[black] (7.5,3.5) -- (7,2) -- (7.5,.5);
    \draw[gray!60, fill = gray!60] (8.575,3.5) -- (7,2) -- (8.429,3.5);
   \draw [sand, pattern=north east lines, pattern color=sand] (9.366,1.729)arc[start angle=46.92, end angle=120, radius=2] --  (7,2) arc [start angle=150, end angle=90, radius=2.732];
   \fill(7,2) circle [radius=.05];
   \draw[blue, ->] (1.5,1.5) -- (.5,2.5) node[above] {$t$}; \draw[red, ->] (1,2)  -- (1.5,2.5) node[above] {$s$};

\end{tikzpicture}
 
		\caption{ A point of order 1. The gray is $0 \le C_{in}^{-1}y \le x \le C_{in} y$, and the black is $|y|=C_{out}x$ with the values  $C_{in}=1.05$, $C_{out}=3$, chosen far from optimal for emphasis. In the left hand side, $t$ is the coordinate along $v$ and $s$ is the coordinate along $v^{\bot}$.
        }\label{f:order}
\end{figure}
We now state our final preliminary definition.
	\begin{definition}
		We say $W \in \Dc^{k,\frac{1}{4}}(\T^2)$ if $W \in W^{k, \infty}(\T^2)$, $W$ is nonnegative,  not identically $0$, and $|\partial^\alpha W| \lesssim W^{1-\frac{|\alpha|}{4}}$ for all multiindices $\alpha$ such that $|\alpha| \le 2$.
	\end{definition}
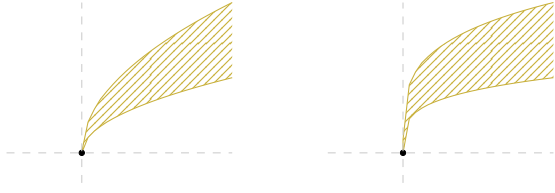
\begin{figure}[h]
    \begin{tikzpicture} 
  \draw[sand, pattern=north east lines,  pattern color=sand]plot [domain=2:0] (\x,{2*(0.5*\x)^.5}) --  plot [domain=0:2] (\x,{(0.5*\x)^.5});
  \filldraw (0,0) circle (1pt);
  \draw[dashed, gray!40] (-1,0) -- (2,0);\draw[dashed, gray!40] (0,-.4) -- (0,2);
   \end{tikzpicture}
   \hspace{1cm}
       \begin{tikzpicture} 
  \draw[sand, pattern=north east lines, pattern color=sand]plot [domain=2:0] (\x,{2*(0.5*\x)^.25})  --  plot [domain=0:2] (\x,{(0.5*\x)^.25});
  \filldraw (0,0) circle (1pt);
  \draw[dashed, gray!40] (-1,0) -- (2,0); \draw[dashed, gray!40] (0,-.4) -- (0,2);
   \end{tikzpicture}
     \caption{With respect to the horizontal direction, a point of order $1/2$ (damping region $\frac{1}{2}x^{1/2} < y < 2x^{1/2} $) and a point of order $1/4$ (damping region $\frac{1}{2} x^{1/4} < y < 2 x^{1/4}$). With respect to the vertical direction, the orders are $2$ and $4$.}
     \label{f:order-fang}
   \end{figure}
    
    Our next theorem is a simplified version of our main result, and implies cases (B) and (C) of  Theorem \ref{t:1ex}. It says that if every point in the glancing set has order $\eta$ and the damping grows like $d^{\beta}$ near $\Gc$, then the energy decay rate is that of a $y$-invariant damping growing polynomially in $d$ with power $\frac{\beta}{\min\{\eta,1\}}+\frac{1}{\eta}$.
    \begin{theorem}\label{t:uniform}
    Let $W \in \Dc^{9,\frac{1}{4}}(\T^2)$. Suppose  there exist $ \eta>0$ and $ \beta \ge 9$ such that
     every $z \in {\mathcal G}$ has order $\eta$,
    and, for $z$ in a neighborhood of ${\mathcal G}$,
    \[
      d(z)^{\beta} \lesssim W(z) \lesssim d(z)^{\beta}.
    \]
    Then \eqref{eq:stable} holds with
	\begin{equation}
		\alpha = 1 - \frac{1}{\frac{\beta}{\min\{\eta, 1\}} + \frac 1 \eta +3}.
	\end{equation} 
    \end{theorem}
    \begin{definition}
    Consider a damping with $W(z) \simeq d(z)^{\beta}$ near $\Gc$. If \eqref{eq:stable} holds with 
    \begin{equation}
        \alpha = 1 - \frac{1}{\beta' +3},
    \end{equation}
    then we refer to $\beta \mapsto \beta'$ as the decay improvement. Theorem \ref{t:uniform} provides a decay improvement of $\beta \mapsto \frac{\beta}{\min\{\eta, 1\}} + \frac 1 \eta$. 
    \end{definition}

    \begin{remark} \hfill 
        \begin{enumerate}
            \item When some geodesics never intersect $\overline{\omega}$, based on \cite[Theorem~1.9]{Kleinhenz2025} we anticipate that the further $\p \omega$ is from its glancing lines, i.e.\ the lower the order of glancing points, the faster the energy will decay. Theorem \ref{t:uniform} bears this out. That is, as $\eta \ra 0^+$ the energy decay rate $\alpha$  increases towards $1$. In this case, by \cite[Theorem 2.5]{AL14} we cannot have $\alpha >1$. On the other hand $\partial \omega$ becomes flat as $\eta \to \infty$, and the energy decay rate is nearly that of $(|x|-\sigma)_+^{\beta}$, namely $\alpha=1 - \frac{1}{\beta+3}$. 
            \item Because sets with nonzero curvature have order 2 at all glancing points, this theorem generalizes the energy decay rate improvement from \cite[Theorem 1.1]{Sun23}. The generalization is in two directions. First, we handle orders other than $2$, which allows us to treat polygons, curvilinear polygons,  and super-ellipses of any order. Second, we show that the behavior of $\p \omega$ and $W$ are relevant only near $\Gc$. This makes it possible to give improved energy decay rates for more general damped regions, and for damping with more general behavior near $\p \omega \cap \Gc^c$.

        \end{enumerate}
    \end{remark}
    
\begin{figure}[h]
    \begin{tikzpicture} 
   \draw[thick] (0,0) rectangle (4,4);
  \draw[sand, pattern=north east lines, pattern color=sand] 
    (3.225,2.707) arc [start angle=30, delta angle=345, radius=1.414] -- (3.366,2.366) arc [start angle=14, delta angle=-343, radius=1.321];
    \draw[dashed, gray!40] (2,0) -- (0,2); \filldraw (1,1) circle (1pt); 
    \draw[dashed, gray!40] (4,2) -- (2,4); \filldraw (3,3) circle (1pt); 
    \draw[dashed, gray!40] (0,2) -- (2,4); \filldraw (1,3) circle (1pt); 
    \draw[dashed, gray!40] (2,0) -- (4,2); \filldraw (3,1) circle (1pt);    
    \draw[dashed, gray!40] (0,3.4142) -- (4,3.4142); \filldraw (2,3.4142) circle (1pt); 
    \draw[dashed, gray!40] (3.4142,0) -- (3.4142,4); \filldraw (3.4142,2) circle (1pt); 
    \draw[dashed, gray!40] (0,.5858) -- (4,0.5858);  \filldraw (2,0.5858) circle (1pt); 
    \draw[dashed, gray!40] (.5858,0) -- (0.5858,4);  \filldraw (0.5858,2) circle (1pt); 

    \end{tikzpicture}
    \hspace{1cm}
  \begin{tikzpicture} 
   \draw[thick] (0,0) rectangle (4,4);
  \draw[sand, pattern=north east lines, pattern color=sand] 
    (3.4142,2) arc (225:180:1.4142) -- (3,3) arc (270:225:1.4142)-- (2,3.4142) arc (315:270:1.4142) -- (1,3) arc (360:315:1.4142) -- (0.586,2) arc(45:0:1.4142) -- (1,1) arc (90:45:1.4142) -- (2,.586) arc(135:90:1.4142) -- (3,1) arc(180:135:1.4142);
    \draw[dashed, gray!40] (2,0) -- (0,2); \filldraw (1,1) circle (1pt); 
    \draw[dashed, gray!40] (4,2) -- (2,4); \filldraw (3,3) circle (1pt); 
    \draw[dashed, gray!40] (0,2) -- (2,4); \filldraw (1,3) circle (1pt); 
    \draw[dashed, gray!40] (2,0) -- (4,2); \filldraw (3,1) circle (1pt);    
    \draw[dashed, gray!40] (0,3.4142) -- (4,3.4142); \filldraw (2,3.4142) circle (1pt); 
    \draw[dashed, gray!40] (3.4142,0) -- (3.4142,4); \filldraw (3.4142,2) circle (1pt); 
    \draw[dashed, gray!40] (0,.5858) -- (4,0.5858);  \filldraw (2,0.5858) circle (1pt); 
    \draw[dashed, gray!40] (.5858,0) -- (0.5858,4);  \filldraw (0.5858,2) circle (1pt); 
\end{tikzpicture}
\hspace{1cm}
 \begin{tikzpicture} 
   \draw[thick] (0,0) rectangle (4,4);
  \draw[sand, pattern=north east lines, pattern color=sand] 
    (3.4142,2) arc (225:180:1.4142) -- cycle; \draw[sand, pattern=north east lines, pattern color=sand]  (3,3) arc (270:225:1.4142) -- cycle; \draw[sand, pattern=north east lines, pattern color=sand]  (2,3.4142) arc (315:270:1.4142) -- cycle; \draw[sand, pattern=north east lines, pattern color=sand] (1,3) arc (360:315:1.4142) -- cycle; \draw[sand, pattern=north east lines, pattern color=sand]  (0.586,2) arc(45:0:1.4142) -- cycle; \draw[sand, pattern=north east lines, pattern color=sand]  (1,1) arc (90:45:1.4142) -- cycle; \draw[sand, pattern=north east lines, pattern color=sand]  (2,.586) arc(135:90:1.4142) -- cycle; \draw[sand, pattern=north east lines, pattern color=sand] (3,1) arc(180:135:1.4142) -- cycle;   \draw[sand, pattern=north east lines, pattern color=sand] 
    (2,2) circle (.1);
    \draw[dashed, gray!40] (2,0) -- (0,2); \filldraw (1,1) circle (1pt); 
    \draw[dashed, gray!40] (4,2) -- (2,4); \filldraw (3,3) circle (1pt); 
    \draw[dashed, gray!40] (0,2) -- (2,4); \filldraw (1,3) circle (1pt); 
    \draw[dashed, gray!40] (2,0) -- (4,2); \filldraw (3,1) circle (1pt);    
    \draw[dashed, gray!40] (0,3.4142) -- (4,3.4142); \filldraw (2,3.4142) circle (1pt); 
    \draw[dashed, gray!40] (3.4142,0) -- (3.4142,4); \filldraw (3.4142,2) circle (1pt); 
    \draw[dashed, gray!40] (0,.5858) -- (4,0.5858);  \filldraw (2,0.5858) circle (1pt); 
    \draw[dashed, gray!40] (.5858,0) -- (0.5858,4);  \filldraw (0.5858,2) circle (1pt); 
\end{tikzpicture}

    \caption{On the left, the decay improvement is $\beta \mapsto \beta + 1/2$, the same as that of the disk of Figure \ref{f:disk}, because the glancing points are the same and all have order $2$. In the middle and on the right, the decay improvement is~$\beta \mapsto \beta + 1$ because the glancing points have order $1$.}\label{f:disk2}
    \end{figure}
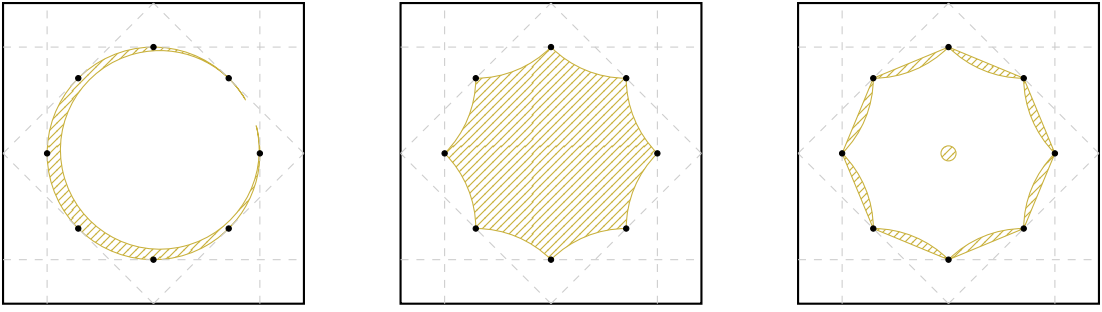


    Theorem \ref{t:uniform} follows from the more general and precise Theorem  \ref{thm:suff} below.
	\begin{theorem}
		\label{thm:suff} 
		Let $W \in \Dc^{9,\frac{1}{4}}(\T^2)$. Suppose that, for each $v\in \mathcal V$, there exist constants $\beta_v$ and $\gamma_v$ such that 
		\begin{enumerate}
			\item 
            for all glancing lines $L \in \mathcal{L}_1(v)$ and all $p \in \Gc(v) \cap L$, there exist $\eta>0$ and $\beta\ge9$ such that $p$ has order $\eta$, for $z$ in a neighborhood of $p$,
			\begin{equation}\label{eq:orderVanish1}
				  d(z)^{\beta} \lesssim W(z) \lesssim d(z)^{\beta},
			\end{equation}
            and 
            \begin{equation}
				\frac{\beta}{\min(\eta, 1)} + \frac{1}{\eta} = \beta_v.
			\end{equation}
			\item 
            for all glancing lines $L \in \mathcal{L}_2(v)$ and all $p \in \Gc(v) \cap L$, there exist $\eta>0$ and $\gamma\ge9$ such that $p$ has order $\eta$, for $z$ in a neighborhood of $p$
			\begin{equation}\label{eq:orderVanish2}
				d(z)^{\gamma} \lesssim W(z) \lesssim d(z)^{\gamma},
			\end{equation}
            and
            \begin{equation}
				\frac{\gamma}{\min(\eta, 1)} + \frac{1}{\eta} = \gamma_v.
			\end{equation}
            
		\end{enumerate} 
		Then \eqref{eq:stable} holds with 
		\begin{equation}
			\alpha=1-\frac{1}{\beta'+3}, \qquad \beta' = \min_{v \in \mathcal V} \beta_v.
		\end{equation}
		If moreover 
        $\mathcal{L}_1=\varnothing$, then \eqref{eq:stable} holds with
		\begin{equation}
			\alpha = 1 +\frac{2}{\gamma'}, \qquad \gamma' =\max_{v \in \mathcal V}  \gamma_v.
		\end{equation}
	\end{theorem}
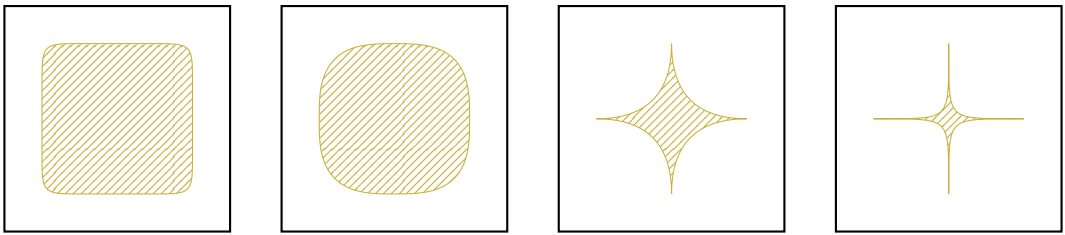
\begin{figure}[h]
       \begin{tikzpicture} 
  \draw[thick] (-1.5,-1.5) rectangle (1.5,1.5); \draw[sand, pattern=north east lines, pattern color=sand, domain=0:pi/2] plot ({cos(\x r)^(2/10)},{sin(\x r)^(2/10)}) -- plot ({-sin(\x r)^(2/10)},{cos(\x r)^(2/10)}) -- plot ({-cos(\x r)^(2/10)},{-sin(\x r)^(2/10)}) -- plot ({sin(\x r)^(2/10)},{-cos(\x r)^(2/10)}) -- cycle; 
   \end{tikzpicture}
    \hspace{.4cm}
    \begin{tikzpicture} 
  \draw[thick] (-1.5,-1.5) rectangle (1.5,1.5); \draw[sand, pattern=north east lines, pattern color=sand, domain=0:pi/2] plot ({cos(\x r)^(2/3)},{sin(\x r)^(2/3)}) -- plot ({-sin(\x r)^(2/3)},{cos(\x r)^(2/3)}) -- plot ({-cos(\x r)^(2/3)},{-sin(\x r)^(2/3)}) -- plot ({sin(\x r)^(2/3)},{-cos(\x r)^(2/3)}) -- cycle; 
   \end{tikzpicture}
    \hspace{.4cm}
    \begin{tikzpicture} 
  \draw[thick] (-1.5,-1.5) rectangle (1.5,1.5); \draw[sand, pattern=north east lines, pattern color=sand, domain=0:pi/2] plot ({cos(\x r)^(2/.5)},{sin(\x r)^(2/.5)}) -- plot ({-sin(\x r)^(2/.5)},{cos(\x r)^(2/.5)}) -- plot ({-cos(\x r)^(2/.5)},{-sin(\x r)^(2/.5)}) -- plot ({sin(\x r)^(2/.5)},{-cos(\x r)^(2/.5)}) -- cycle; 
   \end{tikzpicture}
       \hspace{.4cm}
    \begin{tikzpicture} 
  \draw[thick] (-1.5,-1.5) rectangle (1.5,1.5); \draw[sand, pattern=north east lines, pattern color=sand, domain=0:pi/2] plot ({cos(\x r)^(2/.3)},{sin(\x r)^(2/.3)}) -- plot ({-sin(\x r)^(2/.3)},{cos(\x r)^(2/.3)}) -- plot ({-cos(\x r)^(2/.3)},{-sin(\x r)^(2/.3)}) -- plot ({sin(\x r)^(2/.3)},{-cos(\x r)^(2/.3)}) -- cycle; 
   \end{tikzpicture}
     \caption{Super-ellipses, or balls with respect to the $L^\eta$ metric, for various values of $\eta>0$, decreasing from left to right. As $\eta\to 0$,  the damping set $\omega$ shrinks, while the decay improvement $\beta \mapsto (\beta+1)/\eta$ increases. Conversely, as $\eta \ra \infty$, $\omega$ grows, while the decay improvement $\beta \mapsto \beta+1/\eta$ decreases.
      }
   \end{figure}
	\begin{remark}
    Theorem \ref{thm:suff} further generalizes \cite{Sun23} and \cite[Theorem 1.9]{Kleinhenz2025}:
    \begin{enumerate}
        \item We allow different points of $\Gc$ to have different orders and different  behavior of $W$.
        \item We distinguish one-sided and two-sided glancing lines. If there are one-sided glancing lines then the exact behavior at the two-sided glancing lines is irrelevant. If there are no one-sided glancing lines, then the energy decay rate is faster since $ 1 + \frac{2}{\gamma'} > 1 > 1- \frac{1}{\beta'+3}$.   
        \item When $\mathcal{L}_1=\varnothing$, every geodesic eventually intersects $\overline{\omega}$. In this case, Theorem \ref{thm:suff} extends \cite[Theorem 1.7]{LeautaudLerner2017} and \cite[Example 2.3.2]{Kleinhenz2025} to damping which vanishes on sets larger than finite unions of geodesics. See also \cite{BurqZuily2015,BurqZuily2016} for results on dampings vanishing on a submanifold.
        \item If both ${\mathcal L}_1$ and ${\mathcal L}_2$ are empty   and $\{W>0\}$ is nonempty, then $\{W>0\}$ satisfies the geometric control condition and the energy decays exponentially \cite{RauchTaylor1975}. Exponential energy decay also holds when $\mathcal L_1 = \varnothing$ and $\mathcal L_2 \ne \varnothing$, provided the  damping is a suitable sum of indicator functions of polygons \cite{BurqGerard2018}. 
        \item
		Our proof of this Theorem, combined with \cite{Kleinhenz2025}, should apply to damping  satisfying \eqref{eq:orderVanish1} with $d^{\beta}$ replaced by $d^{\beta} \ln(d^{-1})^{-\rho}$, and give  energy decay with $t^{-\alpha}$ replaced in \eqref{eq:stable} by 
		\begin{equation}
			r(t)= t^{1-\frac{1}{\beta'+3}} \ln(t)^{-\frac{\rho}{\beta'+3}}.
		\end{equation}
		However, for ease of exposition, we focus on the purely polynomial case.
    \end{enumerate}
    \end{remark}
    Before moving on to our genericity results, we outline the proof of this theorem which further clarifies the result. The proof of this theorem relies on a normal form result, Proposition \ref{prop:avgresolve}, that allows us to replace the damping by its averages along glancing directions. Then invoking \cite{DatchevKleinhenz2020} and \cite{LeautaudLerner2017} we reduce proving energy decay rates to obtaining polynomial bounds on the averaged damping near the boundary of its support, Proposition \ref{prop:1dresolve}. We finally show that averaging $d(z)^{\beta}$ near a glancing point of order $\eta$ produces a function vanishing like $d(z)^{\beta_v}$, Proposition \ref{prop:order}. Because of this approach it is correct to interpret $\beta_v$ (and thus $\beta'$) as replacing $\beta$ in the the polynomial energy decay. 

\subsection{Genericity results}
To motivate our study of generic damping, consider $\omega$ given by a square with edges parallel to the edges of the torus and a rotation of this $\omega$ so that its edges all have irrational slope. 

\begin{figure}[h]
    \begin{tikzpicture}
        \draw[sand, pattern=north east lines, pattern color=sand] 
        (-1,-1) -- (1,-1) -- (1,1) -- (-1,1) -- cycle;
       \draw[thick] (-2,-2) rectangle (2,2);\node at (0,-2.6) {(A)};
    \end{tikzpicture}
    \hspace{.5cm}
    \begin{tikzpicture}
        \def\angle{30}
        \draw[sand, pattern=north east lines, pattern color=sand] 
        ({-cos(\angle)+sin(\angle)}, {-sin(\angle)-cos(\angle)}) -- ({-cos(\angle)-sin(\angle)},{-sin(\angle)+cos(\angle)}) -- ({cos(\angle)-sin(\angle)},{sin(\angle)+cos(\angle)}) -- ({cos(\angle)+sin(\angle)},{sin(\angle)-cos(\angle)}) -- cycle;
       \draw[thick] (-2,-2) rectangle (2,2);\node at (0,-2.6) {(B)};
    \end{tikzpicture}
    
    \caption{(A) A square damping set, (B) the same damping set rotated $30$ degrees counterclockwise about its center.}\label{f:squares}
\end{figure}
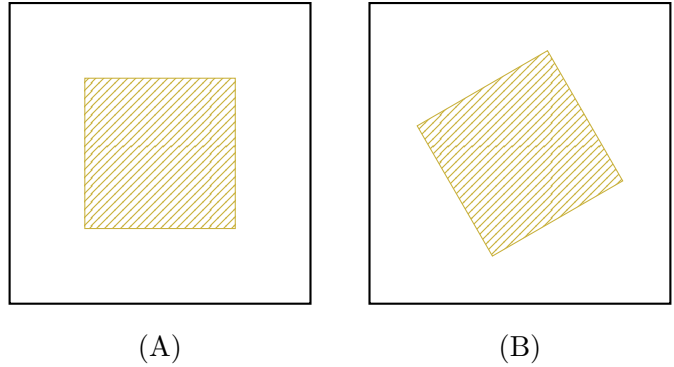
\begin{example}\label{ex:square}
Suppose $W(z) \simeq d(z)^{\beta}$ near $\p \omega$. In case (A) of Figure \ref{f:squares} by \cite[Theorem 1.6]{Kleinhenz2025} we know \eqref{eq:stable} holds with $\alpha = 1- \frac{1}{\beta+3}$. Our result does not provide an improvement because every point in the vertical and horizontal edges is in $\Lc$ and does not have an order. However, in case (B) the only glancing points are the vertices, which have order 1. Thus our Theorem \ref{t:uniform} improves the energy decay rate to $\alpha = 1 - \frac{1}{\beta+1+3}$. This same argument provides an improved decay rate for any rotation of $\omega$ such that its edges have irrational slope, and in fact for all but finitely many rotations of $\omega$ due to Lemma~\ref{l:finite-trapped}. For the remaining rotations of $\omega$ we cannot provide an improvement; an edge parallel to a rational glancing direction may have points in $\Gc$ without an order. Because of this, we say that there are more rotations of $\omega$ which attain the improvement than do not.
\end{example}

We now generalize this to non-degenerate polygons without self-intersections and simple closed $C^2$ curves.

To begin, we give a definition of rotation on the torus. Let $\omega \subset \Rb^2$ be a bounded open set. Denote by $\pi : \Rb^2 \to \Tb^2 \simeq \Rb^2/\Zb^2$ the natural projection on the torus,
and we define a class of sets which behave well with respect to the projection $\pi$.

\begin{definition} \label{def:proj}
We say that a set $\omega \subset \Rb^2$ is properly projected on the torus if $\pi_{\vert \overline{\omega}}$ is one-to-one.
\end{definition}

The decay rate problem that we are investigating is invariant by translation: for any translation $v_0$ the decay rate for the damped wave equation with the damping coefficients $W$ and $W(\bullet - v_0)$ are the same. Therefore, it is natural to identify $\omega$ and all its translations on the torus, or equivalently, $\omega$ and all its translations on $\Rb^2$. Now, denoting by $R_\theta$ the rotation of angle $\theta \in \Sb^1 \simeq \Rb/2 \pi \Zb$ on $\Rb^2$, we have for any $v_0 \in \Rb^2$:
\begin{equation*}
R_\theta (v_0 + \omega)
    = R_\theta v_0 + R_\theta \omega ,
\end{equation*}
namely $R_\theta (v_0 + \omega)$ and $R_\theta \omega$ differ only by a translation. Thus the rotation is well defined on $\Tb^2$, up to translation. We only use this definition for the values of $\theta$ for which $R_\theta \omega$ is properly projected on the torus. 
For future reference, we introduce
\begin{equation*}
\Theta(\omega)
    := \left\{ \theta \in \Sb^1 \;:\; R_\theta \omega \;\, \textrm{is properly projected on the torus} \right\} .
\end{equation*}

This is an open subset of $\Sb^1$.

\begin{remark}
If $\omega$ has diameter $<1$, then $\Theta(\omega) = \Sb^1$.
\end{remark}

\subsubsection{Polygonal damping sets}


We now investigate the case where the damping set is a non-degenerate polygon without self-intersections. First, we introduce some notation to parametrize the ``space of polygons" on a torus.

A general polygon in $\Rb^2$ with $n \ge 3$ vertices is given by an ordered list of vertices $x_1, x_2, \ldots, x_n \in \Rb^2$. Modulo translations, one can describe a polygon by an ordered list of vectors $v_1, v_2, \ldots, v_n \in \Rb^2$, representing the oriented edges, such that
\begin{equation} \label{e:polygon}
v_1 + v_2 + \cdots + v_n = 0 .
\end{equation}
Hence, up to translation, polygons with $n$ vertices (possibly degenerate and self-intersecting) are parametrized by the $(2n-2)$ dimensional subspace $\mathcal{H}_n \subset \Rb^{2n}$ of equation~\eqref{e:polygon}. A polygon is non-degenerate if all its edges have positive length, namely $v_j \neq 0$ for all $1 \le j \le n$.

Rotations $(R_\theta)_{\theta \in \Sb^1}$ on $\Rb^2$ act on $\mathcal{H}_n$ through the mapping
\begin{equation*}
P = (v_1, v_2, \ldots, v_n)
	\longmapsto R_\theta P = (R_\theta v_1, R_\theta v_2, \ldots, R_\theta v_n) .
\end{equation*}

Given a polygon $P$ (modulo translations) described by its oriented edges $v_1 ,v_2, \ldots, v_n$ and without self-intersections, we denote by $\omega_P \subset \Rb^2$ the area enclosed by those edges. We say that $P$ is properly projected on the torus if ${\omega}_P$ is properly projected on the torus, in the sense of Definition~\ref{def:proj}. In that case, we identify $\omega_P \subset \Rb^2$ and its projection $\pi(\omega_P) \subset \Tb^2$.

One can check that the set of non-degenerate polygons without self-intersections that project properly on the torus, is an open subset $\mathcal{P}_n \subset \mathcal{H}_n$. 

We now state our genericity result for polygons. It says that for ``most" polygons, the only glancing points are vertices, which have order 1. Furthermore for a given polygon, all but finitely many rotations of it possess this same property.

\begin{proposition}  \label{p:polygon}
    Let $n \ge 3$. There is an open dense set $\mathcal{Q} \subset \mathcal{P}_n$, such that for any polygon $P \in \mathcal{Q}$, the glancing set ${\mathcal G}$ of $\omega_P$ is contained in the set of vertices of $P$, and all of these glancing points have order $1$.

    In addition, for any $P \in \mathcal{P}_n$, there exist finitely many angles $\theta_1, \theta_2, \ldots, \theta_{j_0} \in \Theta(\omega_P)$ such that $R_{\theta} P \in \mathcal{Q}$ for any $\theta \in \Theta(\omega_P) \setminus \{\theta_1, \theta_2, \ldots, \theta_{j_0}\}$.
\end{proposition}

Combining Proporition \ref{p:polygon} with Theorem~\ref{thm:suff} yields the following decay improvement, which generalizes Example \ref{ex:square}. 


\begin{corollary} \label{c:polygon}
Let $P \in \mathcal{Q}$, given by Proposition~\ref{p:polygon}. Suppose $\omega = \omega_P$. Then for any $\beta \ge 9$ and $W \in \Dc^{9, \frac{1}{4}}(\Tb^2)$ such that
\begin{equation} \label{e:W-d(z)}
\exists C > 0 : \qquad 
	C^{-1} d(z)^\beta
		\le W(z)
		\le C d(z)^\beta,
\end{equation}
in a neighborhood of the glancing set $\mathcal{G}$ of $\omega_P$,
we have decay at rate
\begin{equation} \label{e:decay-polygon}
\alpha
	=1- \dfrac{1}{\beta +1 + 3} .
\end{equation}
For any $P \in \mathcal{P}_n$, the above applies to all but a finite number of rotations $R_\theta P$.
\end{corollary}

These results are proved in Section~\ref{sec:generic-polygon}.

\subsubsection{$C^2$ damping sets}

In this section, we consider damping sets $\omega_\gamma$ whose boundary is given by a simple closed $C^2$ curve $\gamma$. Denote by $\mathcal{U}$ the set of simple closed curves $\gamma \in C^2(\Sb^1; \Rb^2)$, such that $\omega_\gamma$ is properly projected on the torus and $\dot \gamma(t) \neq 0$ for all $t \in \Sb^1$. One can check that this is an open subset of $C^2(\Sb^1; \Rb^2)$ for the $C^2$ topology. 
Notice that for any $\gamma \in \mathcal{U}$, the curvature
\begin{equation} \label{e:def-curvature}
\kappa\left(\gamma(t)\right)
    := \dfrac{1}{|\dot \gamma(t)|^3} \sqrt{|\ddot \gamma(t)|^2 |\dot \gamma(t)|^2 - (\ddot \gamma(t) \cdot \dot \gamma(t))^2} ,
\end{equation}
is well defined since  $\dot \gamma(t) \neq 0$ for all times, and it does not depend on the parametrization of the curve. As we did for polygons, we can rotate $\gamma$, up to translations, by setting
\begin{equation*}
(R_\theta \gamma)(t)
    := R_\theta\left( \gamma(t) \right) ,
        \qquad t \in \Sb^1 .
\end{equation*}


    We now state our genericity result. It says that, for ``most" simple closed $C^2$ curves, all glancing points have nonvanishing curvature and so have order 2. Furthermore, for a given curve, all but a compact measure zero set of rotations of it possess this same property.
    \begin{proposition}\label{p:curve}
    There is an open dense subset $\mathcal{Y} \subset \mathcal{U}$ such that for any $\gamma \in \mathcal{Y}$, the glancing set $\mathcal{G}$ of $\omega_\gamma$ is contained in the subset of $\gamma$ where the curvature $\kappa > 0$. In particular the glancing points all have order 2. 

    In addition, for any $\gamma \in \mathcal{U}$, there exists a compact set $K \subset \Sb^1$ of measure zero such that $R_\theta \gamma \in \mathcal{Y}$ for all $\theta \in \Theta(\omega_\gamma) \setminus K$.
\end{proposition}

As a consequence of Theorem~\ref{thm:suff}, we have that if the damping set is bounded by a curve, then for most simple closed $C^2$ curves the energy decay rate is improved relative to damping supported on a strip and growing at the same rate. Moreover, given a fixed damping set of this form, ``most" rotations of this set exhibit this improved decay rate.

\begin{corollary} \label{c:curve}
Let $\gamma \in \mathcal{Y}$, given by Proposition~\ref{p:curve}. Suppose $\omega = \omega_{\gamma}$. Then for any $\beta \ge 9$ and $W \in \Dc^{9, \frac{1}{4}}(\Tb^2)$ such that
\begin{equation*}
\exists C > 0 : \forall z \in \Tb^2, \qquad 
	C^{-1} d(z)^\beta
		\le W(z)
		\le C d(z)^\beta,
\end{equation*}
we have decay at rate
\begin{equation*}
\alpha
	= 1-\dfrac{1}{\beta + \frac{1}{2} + 3} .
\end{equation*}

For any $\gamma \in \mathcal{U}$, the above applies to rotations $R_\theta \gamma$ with $\theta \in \Theta(\omega_\gamma) \setminus K$, where $K$ is a compact subset of $\Sb^1$ with measure zero. 
\end{corollary}

These results are proved in Section~\ref{sec:generic-curve}.

\begin{remark}\hfill 
\begin{enumerate}
    \item 
Recall that compact sets of measure zero are nowhere dense.	
\item This corollary generalizes the energy decay rate result of \cite{Sun23}, which obtains the same energy decay rate, but requires the curvature of $\gamma$ to be positive everywhere.
\end{enumerate}
\end{remark}
Although for polygons the number of rotations that do not produce an improved decay rate is finite, a simple closed  curve  $\gamma \in C^2(\Sb^1; \Rb^2)$ can have $K$ uncountable. For example $K$ can be the Cantor set:
	\begin{example}  Define $f \in C^\infty([0,1])$ by $f(x)= x + \int_0^x \int_0^y \phi(z) \,dz \,dy$, where $\phi$ is a nonnegative $C^{\infty}$ function whose zero set equals the Cantor set. Then $f'>0$, $f'' \geq 0$, and $f''$ vanishes on the Cantor set. Take $\gamma \in \mathcal U$ such that $\omega_\gamma$ is convex and contained in a small enough disk so that $\Theta(\omega_\gamma) = \mathbb S^1$, and such that a subset of $\gamma(\mathbb S^1)$ is similar to the graph of $f$. By convexity, every tangent line of the graph of $f$ is locally a one-sided glancing line, and $t \mapsto f'(t) $ is injective. Thus,   uncountably many  directions  have a tangent point with zero curvature. Then each rotation of $\omega_\gamma$ that makes such a  direction parallel to $(1,0)$ is an element of $K$. 
	\end{example}





\textbf{Acknowledgments:} The first author was supported by a Simons collaboration grant for mathematicians.
The second author was supported by the National Science Foundation [DMS-2530465].

\section{Reduction to Averaging}	

By \cite[Theorem 2.4]{BorichevTomilov2010}, as stated in \cite[Proposition 2.4]{AL14}, the wave stabilization estimate \eqref{eq:stable} is equivalent to the resolvent estimate
\begin{equation}\label{e:resest}
			\nm{\left(-\Delta+ i\lambda W - \lambda^2 \right)^{-1}}_{\Lc(L^2(\Tb^2))} \lesssim \lambda^{\frac{1-\alpha}\alpha},  \qquad \text{for }\lambda \gg 1.
\end{equation} 

We reduce \eqref{e:resest} to a family of one dimensional estimates by averaging over glancing directions $v \in \mathcal V$. For $v \in \mathcal V$, define the averaging operator along $v$:
	\begin{equation}
		f \mapsto A(f)_v(s) = \frac{1}{T_v} \int_0^{T_v} f(sv^{\bot}+tv) dt,
	\end{equation}
	where $T_v$ is the period of $t \mapsto s v^{\bot} +tv$.
	Since $v$ is rational, using a standard change of coordinates we may regard $A(f)_v(s)$ as a function on $\Sb^1$. See \cite[Section 6]{AL14}, \cite[Section 2.2]{Sun23}, or \cite[Section 5.3]{Kleinhenz2025}.

	Now we can state the normal form result that relates resolvent estimates for $y$-invariant damping and general damping. Roughly, it says that for sufficiently regular damping $W$, if the average of $W$ along every direction $v \in \mathcal V$ produces a resolvent estimate, then $W$ produces the same estimate. 
	This is \cite[Theorem 1.12]{Kleinhenz2025}:
	\begin{proposition}\label{prop:avgresolve}
		Suppose $W \in \Dc^{9,\frac{1}{4}}(\T^2)$, and there exists $\rho:[1,\infty) \ra (0,\infty)$ with $\pl = o(\lambda^{\frac{1}{3}})$, such that  for all $v \in \mathcal V$, there exist $\lambda_v, C_v>0$, such that for $\lambda \geq \lambda_v$ and all $E \in \mathbb R$. 
		\begin{equation}\label{assumed1dresolventeq}
			\nm{\left(-\p_s^2 + i\lambda A_v(W)(s) -E\right)^{-1} }_{\Lc(L^2(\Sb^1))} \leq C_v \pl, 
		\end{equation}
		then there exists $C, \lambda_0>0$ such that for $\lambda \geq \lambda_0$ 
		\begin{equation}
			\nm{\left(-\Delta+ i\lambda W - \lambda^2 \right)^{-1}}_{\Lc(L^2(\Tb^2))} \leq C \pl. 
		\end{equation}
	\end{proposition}
	To obtain a 1-d resolvent estimate of the form \eqref{assumed1dresolventeq} for $v \in \mathcal V$,  we combine the estimates of \cite{DatchevKleinhenz2020} and \cite{LeautaudLerner2017}. 
	\begin{proposition}\label{prop:1dresolve}
		Suppose $V(s) \in C^{0}(\Sb^1)$ and $V(s)=0$ only on finitely many intervals $[a_j,b_j]$ and at finitely many points $s_j$. Suppose there exist $C_0 \geq 1$ and $\beta_j, \gamma_j >0$ such that 
		\begin{enumerate}
            \item For $s \in \{V>0\}$ in a neighborhood of $[a_j,b_j]$ ,
			\begin{equation}
				C_0^{-1} d_j(s)^{\beta_j} \leq V(s) \leq C_0 d_j(s)^{\beta_j},
			\end{equation}
			where $d_j(s)=\dist(s,[a_j, b_j])$.
			\item For $s \in \{V>0\}$ in a neighborhood of $s_j$,
			\begin{equation}
				C_0^{-1}|s-s_j|^{\gamma_j} \leq V(s) \leq C_0 |s-s_j|^{\gamma_j}.
			\end{equation}
		\end{enumerate}
		Let $\beta'= \min \beta_j$, then there exist $\lambda_0, C>0$, such that for all $E \in \Rb$ and $\lambda \geq \lambda_0$,
		\begin{equation}\label{e:1dbetaest}
			\nm{(-\p_s^2 + i \lambda V(s) - E)^{-1}}_{\Lc(L^2(\Sb^1))} \leq C \lambda^{\frac{1}{\beta'+2}}. 
		\end{equation}
		Furthermore, if $V$ vanishes only at points $s_j$, let $\gamma' = \max \gamma_j$. Then there exist $\lambda_0, C>0$, such that for all $E \in \Rb$ and $\lambda \geq \lambda_0$,
		\begin{equation}
			\nm{(-\p_s^2 + i \lambda V(s) - E)^{-1}}_{\Lc(L^2(\Sb^1))} \leq C \lambda^{-\frac{2}{\gamma'+2}}. 
		\end{equation}
	
	\end{proposition}
	\begin{remark}
		Note that 
		\begin{equation}
			-\frac{2}{\gamma'+2}<0 < \frac{1}{\beta'+2}.
		\end{equation}
		So a damping vanishing only at points satisfies a stronger resolvent estimate than any damping vanishing on an interval. This is why \eqref{e:1dbetaest} is independent of $\gamma'$. 
	\end{remark}

%
	
	In light of the above propositions, to prove Theorem \ref{thm:suff} it is enough to show that averaging $W$ along directions $v \in \mathcal V$ improves the polynomial power on $d$ from $\beta$ (resp.\ $\gamma$) to $\beta_v \leq \beta'$ at glancing points along one-sided glancing lines, (resp.\ to $\gamma_v \geq \gamma'$ at glancing points along two-sided glancing lines). That is, it will be enough to prove the following proposition which controls the average of the damping function $W$ along directions $v \in \mathcal V$.

	\begin{proposition}\label{prop:order}
		Under the  assumptions of Theorem \ref{thm:suff}, and recalling the definitions of $\gamma, \gamma_v, \beta,$ and $ \beta_v$ there, for each $v \in \mathcal V$, there exists $C_0 \geq 1$ such that 
		\begin{enumerate}
			\item $A_v(W)(s)=0$ only on finitely many intervals $[\alpha_j,\rho_j]$ and at finitely many points $s_j$. Furthermore, if $\Lc_1(v)=\varnothing$, then $A_v(W)$ vanishes only at points $s_j$.
			\item There exists a neighborhood of each $[\alpha_j,\rho_j]$ such that
			\begin{equation}\label{eq:ordergrowavg2}
				C_0^{-1} d_j(s)^{\beta_v} \leq A_v(W)(s) \leq C_0 d_j(s)^{\beta_v},
			\end{equation}
			where $d_j(s)= \dist(s, [\alpha_j,\rho_j])$.
            \item  There exists a neighborhood of each $s_j$ such that
			\begin{equation}\label{eq:ordergrowavg1}
				C_0^{-1} |s-s_j|^{\gamma_v} \leq A_v(W)(s) \leq  C_0 |s-s_j|^{\gamma_v}.
			\end{equation}

		\end{enumerate}

	\end{proposition}
	With these propositions we now prove Theorem \ref{thm:suff}.
		\begin{proof}[Proof of Theorem \ref{thm:suff}]
			We separately address the cases where $\mathcal{L}_1$ is nonempty or empty.
			 
			1)  If   $\mathcal{L}_1 \ne \varnothing$ , then Proposition \ref{prop:order} shows that for each $v \in \mathcal V$, $A_v(W)$ satisfies the hypotheses of Proposition \ref{prop:1dresolve} with  $\gamma_j=\gamma_v$ and $\beta_j=\beta_v$. Therefore for all $v \in \mathcal V$, there exist $C_v, \lambda_v >0$, such that for $\lambda \geq \lambda_v$, and for all $E \in \Rb$,
					\begin{equation}
				\nm{(-\p_s^2 + i \lambda A_v(W)(s) - E)^{-1}}_{\Lc(L^2)} \leq C_v \lambda^{\frac{1}{\beta_v+2}} \leq C_v \lambda^{\frac{1}{\beta'+2}}, 
			\end{equation}
            where the final inequality holds because  $\min \beta_v = \beta'$. 
			Then, by Proposition \ref{prop:avgresolve}, \eqref{e:resest} holds with $\alpha = 1 - \frac{1}{\beta'+3}$, and this completes the proof.
			
			2) If    $\mathcal{L}_1 = \varnothing$, then Proposition \ref{prop:order} shows that for all $v \in \mathcal V$, $A_v(W)$ satisfies the hypotheses of the second part of Proposition \ref{prop:1dresolve} with  $\gamma_j=\gamma_v$. Therefore for all $v \in \mathcal V$, there exist $C_v, \lambda_v >0$, such that for $\lambda \geq \lambda_v$, and for all $E \in \Rb$,
			\begin{equation}
				\nm{(-\p_s^2 + i \lambda A_v(W)(s) - E)^{-1}}_{\Lc(L^2)} \leq C_v \lambda^{-\frac{2}{\gamma_v+2}} \leq C_v \lambda^{-\frac{2}{\gamma'+2}}, 
			\end{equation}
            where the final inequality holds because $\max \gamma_v=\gamma'$.
			Then, by Proposition \ref{prop:avgresolve}, \eqref{e:resest} holds with $\alpha = 1+\frac{2}{\gamma'}$, and this completes the proof.
		\end{proof}

	\section{Proof of 1d combination results, Proposition \ref{prop:1dresolve}}
	For the proof we need the following two consequences of a pairing argument.
	\begin{lemma}\label{l:pairing}
		For $u \in L^2(\Sb^1)$ and $\lambda, E \in \Rb$ define 
		\begin{equation}
			f:=(-\p_s^2+i\lambda V(s)-E)u.
		\end{equation}
		For any $E \in \Rb, \lambda>0$, then
		\begin{equation}
			\ltwo{V^{1/2}u} \leq \lambda^{-\frac{1}{2}} |\<f,u\>|^{1/2}.
		\end{equation}
		Also, for any nonnegative $\psi \in C^{\infty}(\Sb^1)$ which vanishes on a neighborhood of $\{V=0\}$, there exists a $C>0$, such that for any $E \in \Rb, \lambda>0$ then
		\begin{equation}
			\ltwo{\psi^{1/2} \p_s u} \leq C(1+\max(0,E)^{1/2}) \lambda^{-\frac{1}{2}} |\<f,u\>|^{1/2} + C \ltwo{f}.
		\end{equation}
	\end{lemma}
	\begin{proof}
		The first part follows by multiplying $(-\p_s^2+i\lambda V(s)-E)u=f$ by $\bar{u}$. Then integrate by parts and take the imaginary part.
		
		To prove the second part we consider the left hand side, integrate by parts twice, and use the equation to obtain
		\begin{align}
			\int \psi|u'|^2 ds  &= - \Re \int \psi' u' \bar{u} ds  - \Re \int \psi u'' \bar{u} ds \\
			&=\frac{1}{2} \int \psi''|u|^2 + E \int \psi |u|^2  ds + \Re \int \psi f \bar{u} ds.
		\end{align}
		Now use that $\psi \leq CV$, $|\psi''| \le CV$, and H\"older's inequality to write 
		\begin{equation}
			\int \psi|u|^2 ds \leq C (1+\max(0,E)) \int V |u|^2 ds + C \left(\int |f|^2 ds \right)^{1/2} \left( \int V |u|^2 ds \right)^{1/2}  \nm{\psi}_{L^\infty}^{1/2} .
		\end{equation} 
		Then apply Young's inequality for products to the second term, then apply part 1 and take square roots to conclude.
	\end{proof}
	
	We now prove Proposition \ref{prop:1dresolve}. The idea of the proof is to use a partition of unity to study damping only vanishing on individual intervals $[a_j, b_j]$ or points $s_j$ and then obtain resolvent estimates using \cite{DatchevKleinhenz2020} or \cite{LeautaudLerner2017} and a standard 1-d propagation estimate. 
	\begin{proof}[Proof of Proposition \ref{prop:1dresolve}]

	1) Order the finite intervals $[a_j,b_j]$ and zeroes $s_j$ in a single list of length $N$. Then let $\chi_j$ be a partition of unity on $\Sb^1$, such that $\chi_k$ is identically 1 on the $k$th element of the list $\{ [a_j,b_j], s_j\}$, and for some $\e>0$, $\supp \chi_{N+1} \subset \{V \geq \e\}$.
	
	Let $u,f$ solve $(-\p_s^2+i\lambda V(s)-E)u=f$. Then, by Lemma \ref{l:pairing},
	\begin{equation}
		\ltwo{\chi_{N+1} u} \leq \frac{1}{\e^{1/2}} \ltwo{V^{1/2} u} \leq \frac{C}{{\sqrt{\lambda}}} |\<f,u\>|^{1/2}. 
	\end{equation}
	Applying Young's inequality for products, we have for any $\d>0$
	\begin{equation}
		\ltwo{\chi_{N+1}u} \leq \frac{C}{{\lambda} \d} \ltwo{f} + C\d \ltwo{u}.
	\end{equation}
	
	We now separately consider those $j$ associated to intervals $[a_j, b_j]$ or single points $s_j$. Each of these cases will be further split into sub-cases based on the size of $E$. The technique is the same for intervals and points, but the constants involved are different. 
	
	Before doing so we make a common definition. 
	\begin{equation}
		V_j(s) = \begin{cases}
			V(s), \quad & \text{ for } s\in \supp(\chi_j),\\
			\e, \quad &\text{ otherwise}.
		\end{cases}
	\end{equation}
	Note that $V \chi_j = V_j \chi_j$, and if $s_j \in \supp \chi_j$ then $V_j$ satisfies the hypotheses of \cite{LeautaudLerner2017}. Additionally if $[a_j,b_j] \subset \supp \chi_j$, then $V_j$ satisfies the hypotheses of \cite{DatchevKleinhenz2020}. Furthermore 
	\begin{align}
		(-\p_s^2 + i \lambda V_j(s) -E) \chi_j u &= \chi_j(-\p_s^2 + i \lambda V (s) -E) u + [-\p_s^2, \chi_j]u \\
		&= \chi_j f - \chi_j'' u - 2\chi_j' \p_s u.\label{eq:1dchijSDWE}
	\end{align}
    2) For intervals $[a_j, b_j]$ we consider separately the cases $E \leq \lambda^{\frac{\beta'+1}{\beta'+2}}$ and $E \geq \lambda^{\frac{\beta'+1}{\beta'+2}}$.
	
	2a) Assume $E \leq \lambda^{\frac{\beta'+1}{\beta'+2}}$. 
 Then by \cite[equation (10)]{DatchevKleinhenz2020}, there exists $C>0$ such that
	\begin{equation}
		\ltwo{\chi_j u} \leq C \lambda^{\frac{1}{\beta_j+2}}\ltwo{\chi_j f {-} \chi_j'' u {-} 2\chi_j' \p_s u}.
	\end{equation}
		Now since $\beta'=\min \beta_j$, note that $\lambda^{\frac{1}{\beta_j+2}} \leq \lambda^{\frac{1}{\beta'+2}}$. Therefore applying Lemma \ref{l:pairing}, since $|\chi_j''| \leq CV$ and $\chi_j'=0$ on a neighborhood of $V=0$, we have
	\begin{align}
		\ltwo{\chi_j u} &\leq C \lambda^{\frac{1}{\beta'+2}} \left(\ltwo{\chi_j f} + \ltwo{\chi_j'' u} + 2 \ltwo{\chi_j' \p_s u} \right)\\
		&\leq C \lambda^{\frac{1}{\beta'+2}}\ltwo{f} + C (\lambda^{\frac{1}{\beta'+2}-\frac{1}{2}} + \lambda^{\frac{1}{\beta'+2}-\frac{1}{2}} \max(0,E)^{1/2})|\<f,u\>|^{1/2}\\
		&\leq C \lambda^{\frac{1}{\beta'+2}} \ltwo{f} + C \lambda^{\frac{1}{2(\beta'+2)}} |\<f,u\>|^{1/2} ,
	\end{align}
	where the third inequality uses that $E \leq \lambda^{\frac{\beta'+1}{\beta'+2}}$ and that $\lambda$ is large. Then using Young's inequality for products we have, for any $\d>0$
	\begin{align}
		\ltwo{\chi_j u} &\leq \frac{C}{\delta} \lambda^{\frac{1}{\beta'+2}} \ltwo{f} + \d \ltwo{u}.\label{eq:intervalcombineEsmall}
	\end{align}
	2b) Assume $E \geq \lambda^{\frac{\beta'+1}{\beta'+2}}$. 
    Rewrite equation \eqref{eq:1dchijSDWE} as 
	\begin{equation}
		(-\p_s^2-E) \chi_j u= \chi_j f + \chi_j'' u - 2 \p_s(\chi_j' u) - i\lambda V \chi_j u.
	\end{equation}
    Then by a standard 1-d propagation estimate, see \cite[Lemma 6.6]{Sun23} or \cite[Prop 4.2]{Burq2020}, and using that $|\chi_j'| + |\chi_j''| \leq CV$, there exists $C>0$ such that 
	\begin{align}
		\ltwo{\chi_j u} &\leq  C E^{-1/2} \ltwo{f+\chi_j'' u  - i\lambda V \chi_j u} + \hp{\p_s (\chi_j' u)}{-1} + \ltwo{Vu}\\
		&\leq C E^{-1/2} \ltwo{f} + C(\lambda E^{-1/2} +1) \ltwo{V^{1/2} u}.\label{eq:highE1dest}
	\end{align} 
    Note that to obtain \eqref{eq:highE1dest} we did not use the form of $V$ near $[a_j, b_j]$ nor the exact size of $E$.
    
	Now since $E \geq \lambda^{\frac{\beta'+1}{\beta'+2}}$, $E^{-1/2} \leq \lambda^{-\frac{\beta'+1}{2(\beta'+2)}}$, and applying Lemma \ref{l:pairing} we obtain 
	\begin{align}
		\ltwo{\chi_j u} & \leq C \ltwo{f} + C \lambda^{1-\frac{\beta'+1}{2(\beta'+2)}} \ltwo{V^{1/2}u}\\
		&\leq C \ltwo{f} + C \lambda^{\frac{1}{2(\beta'+2)}} |\<f,u\>|^{1/2}.
	\end{align}
	Then applying Young's inequality for products, for any $\d>0$ we have 
	\begin{equation}\label{eq:intervalcombineElarge}
		\ltwo{\chi_j u} \leq \frac{C}{\d} \lambda^{\frac{1}{\beta'+2}} \ltwo{f} + \d \ltwo{u}.
	\end{equation}
	2c) Now by \eqref{eq:intervalcombineEsmall} and \eqref{eq:intervalcombineElarge} for all $E \in \Rb$, if $V$ vanishes on $[a_j, b_j]$, then, for any $\d>0$,
	\begin{equation}
		\ltwo{\chi_j u} \leq \frac{C}{\d} \lambda^{\frac{1}{\beta'+2}} \ltwo{f} + \d \ltwo{u}. \label{eq:intervalchijest}
	\end{equation}
	3) For single points $s_j$ we consider separately the cases $E \leq \lambda^{\frac{\gamma'+4}{\gamma'+2}}$, and $E \geq \lambda^{\frac{\gamma'+4}{\gamma'+2}} $.
	
	3a) Assume $E \leq \lambda^{\frac{\gamma'+4}{\gamma'+2}}$. Applying \cite[Theorem 1.7]{LeautaudLerner2017} to \eqref{eq:1dchijSDWE} there exists $C>0$
	\begin{equation}
		\ltwo{\chi_j u} \leq C \lambda^{-\frac{2}{\gamma_j+2}} \ltwo{\chi_j f {-} \chi_j'' u  {-} 2\chi_j' \p_s u}
	\end{equation}
	Now because $\gamma'=\max \gamma_j$, we have $\lambda^{-\frac{2}{\gamma_j+2}} \leq \lambda^{-\frac{2}{\gamma'+2}}$. Therefore applying Lemma \ref{l:pairing}, since $|\chi_j''| \leq CV$ and $\chi_j'=0$ on a neighborhood of $\{V=0\}$, we have
	\begin{align}
		\ltwo{\chi_j u} &\leq C \lambda^{-\frac{2}{\gamma'+2}} \left(\ltwo{\chi_j f} + \ltwo{\chi_j'' u} + 2 \ltwo{\chi_j' \p_s u} \right)\\
		&\leq C \lambda^{-\frac{2}{\gamma'+2}}\ltwo{f} + C (\lambda^{-\frac{2}{\gamma'+2}-\frac{1}{2}} + \lambda^{-\frac{2}{\gamma'+2}-\frac{1}{2}} \max(0,E)^{1/2})|\<f,u\>|^{1/2}\\
		&\leq C \lambda^{-\frac{2}{\gamma'+2}} \ltwo{f} + C \lambda^{-\frac{1}{\gamma'+2}} |\<f,u\>|^{1/2}.
	\end{align}
	Where the third inequality uses that $E \leq \lambda^{\frac{\gamma'+4}{\gamma'+2}}$. Then using Young's inequality for products we have, for any $\d>0$
	\begin{align}
		\ltwo{\chi_j u} &\leq \frac{C}{\delta} \lambda^{-\frac{2}{\gamma'+2}} \ltwo{f} + \d \ltwo{u}.\label{eq:pointcombineEsmall}
	\end{align}
	3b) Assume $E \geq \lambda^{\frac{\gamma'+4}{\gamma'+2}}$. 
    Following the same argument as in case 2b) we obtain \eqref{eq:highE1dest}
	\begin{align}
		\ltwo{\chi_j u} &\leq C E^{-1/2} \ltwo{f} + C(\lambda E^{-1/2} +1) \ltwo{V^{1/2} u}.
	\end{align} 
	Now since $E \geq \lambda^{\frac{\gamma'+4}{\gamma'+2}}$, $E^{-1/2} \leq \lambda^{-\frac{\gamma'+4}{2(\gamma'+2)}}\leq \lambda^{-\frac{2}{\gamma'+2}}$, and applying Lemma \ref{l:pairing} we obtain 
	\begin{align}
		\ltwo{\chi_j u} & \leq \lambda^{-\frac{2}{\gamma'+2}}\ltwo{f} + C \lambda^{1-\frac{\gamma'+4}{2(\gamma'+2)}} \ltwo{V^{1/2}u}\\
		&\leq C \lambda^{-\frac{2}{\gamma'+2}} \ltwo{f} + C \lambda^{-\frac{1}{\gamma'+2}} |\<f,u\>|^{1/2}.
	\end{align}
	Then applying Young's inequality for products, for any $\d>0$ we have 
	\begin{equation}\label{eq:pointcombineElarge}
		\ltwo{\chi_j u} \leq \frac{C}{\d} \lambda^{-\frac{2}{\gamma'+2}} \ltwo{f} + \d \ltwo{u}.
	\end{equation}
	3c) Combining \eqref{eq:pointcombineEsmall} and \eqref{eq:pointcombineElarge}, for all $E \in \Rb$ if $V$ vanishes at $s_j$ then for any $\d>0$ we have
	\begin{equation}
		\ltwo{\chi_j u} \leq \frac{C}{\d} \lambda^{-\frac{2}{\gamma'+2}} \ltwo{f} + \d \ltwo{u}. \label{eq:pointchijest}
	\end{equation}
	4) We now combine the cases together to obtain the two resolvent estimates. First, we assume that $V$ vanishes on at least one interval $[a,b]$.
	Combining \eqref{eq:intervalchijest} and \eqref{eq:pointchijest}, and noting that $\lambda^{-\frac{2}{\gamma'+2}}\leq \lambda^{\frac{1}{\beta'+2}}$, then we have for any $\d>0$
	\begin{align}
		\ltwo{u} &\leq \sum_{j=1}^N \ltwo{\chi_j u} + \ltwo{\chi_{N+1} u} \\
		&\leq \frac{C}{\d} (\lambda^{\frac{1}{\beta'+2}} + \lambda^{-\frac{2}{\gamma'+2}} )\ltwo{f} + C \d \ltwo{u}\\		
		&\leq \frac{C}{\d} \lambda^{\frac{1}{\beta'+2}} \ltwo{f} + C \d \ltwo{u}.
	\end{align}
	Choosing $\d>0$ small enough, we can absorb the final term on the right-hand side back to obtain the first resolvent estimate. 
	
	Now assume that $V$ vanishes only at points $s_j$. Then applying \eqref{eq:pointchijest},  for any $\d>0$,
	\begin{align}
			\ltwo{u} &\leq \sum_{j=1}^N \ltwo{\chi_j u} + \ltwo{\chi_{N+1} u} \\
			&\leq \frac{C}{\d} \lambda^{-\frac{2}{\gamma'+2}} \ltwo{f} + C \d \ltwo{u}.
	\end{align}
	Choosing $\d>0$ small enough, we can absorb the final term on the right-hand side back to obtain the second resolvent estimate. 
	\end{proof}

    \section{Proof of Proposition \ref{prop:order}}

    Our approach is related to that of \cite[Proposition 4.4]{Sun23} and \cite[Lemmas 6.1, 6.2, 6.3]{Kleinhenz2025}, but our geometric setup is more general.  The results \cite[Proposition 4.4]{Sun23} and \cite[Lemma 6.2]{Kleinhenz2025} apply only to $\omega$ locally strictly convex with positive curvature, which is a special case of all points of $\mathcal G$ having order 2. The results in \cite[Lemmas 6.1, 6.3]{Kleinhenz2025} apply to $\omega$ exactly equal to a rectangle or a super-ellipse with $\eta \geq 2$.

    We use $(s,t)$ coordinates to represent a point $z=sv^{\bot}+tv$ on $\Tb^2$. We  write 
    \begin{equation}
        W(s,t)=W(sv^{\bot}+tv).
    \end{equation}
    Recall the definition of order, Definition \ref{def:order}; for each point $z_0 \in \mathcal{G}(v)$, we can relate $(s,t)$ coordinates on $\T^2$ (where $(s_0,t_0)=z_0$) to $(x,y)$ coordinates in $\Rb^2$ via $\psi$. 

    Equation \eqref{e:psil} implies that $\psi'(z)$ is a lower triangular matrix, i.e.\ the directional derivative of $x$ in the direction $v$ is zero; $D_v x = 0$.
    Since $\psi$ is affine, we can write 
    \begin{equation}\label{eq:psicoords}
            \psi(s,t)
            = \begin{pmatrix}
                a(s-s_0) \\
                b (s-s_0)+c(t-t_0) \end{pmatrix}
            = \begin{pmatrix} x \\ y \end{pmatrix} .
    \end{equation}
    We now prove the three parts of Proposition \ref{prop:order} one at a time.
    \begin{proof}[Proof of Proposition \ref{prop:order}(1)]
    1) Recall the definition of $A_v(W)$
    \begin{equation}
        A_v(W)(s) = \frac{1}{T_v} \int_0^{T_v} W(s,t) dt.
    \end{equation}
    Thus if $s \in \p \{A_v(W)>0\}$, then for some $t$,  we have $sv^{\bot}+tv \in {\mathcal G}(v)$. By Remark \ref{r:finiteG}, ${\mathcal G}$ is finite, and so there are finitely many points in $\p \{A_v(W)>0\}$. Therefore $\{A_v(W)=0\}$ has finitely many connected components. Furthermore $\{A_v(W)=0\}$ is closed as the level set of a continuous function. The only closed connected sets in $\Sb^1$ are closed intervals or points. Thus $\{A_v(W)=0\}$ is a finite union of closed intervals and points. 
    
    2) Now we show that if ${\mathcal L}_1(v) = \varnothing$, then $A_v(W)=0$ only on a finite set. 
    To see this, consider $s_0 \in \p\{A_v(W)>0\}$, so for some $t_0$, we have $t_0 v +s_0 v^{\bot}\in {\mathcal G}(v)$. If 
    \begin{equation}
        A_v(W)(s)=0 \text{ for } s \in (s_0, s_0+\e_0)\quad (\text{or } s\in (s_0-\e_0,s_0)),
    \end{equation}
    for some $\varepsilon_0>0$, then for all $t$
    \begin{equation}
        W(s,t) = 0 \text{ for } s \in (s_0, s_0+\e_0) \quad (\text{or } s\in (s_0-\e_0,s_0)).
    \end{equation}
    That is, for all $\e \in (0,\e_0)$
    \begin{equation}
        L_{t_0v+(s_0+\e) v^{\bot}, v} \cap \omega =\varnothing, \quad (\text{or  }L_{t_0v+(s_0-\e) v^{\bot}, v} \cap \omega =\varnothing).
    \end{equation}
    Therefore $L_{t_0v+s_0 v^{\bot},v} \in {\mathcal L}_1(v) =\varnothing$, which is a contradiction. Thus $A_v(W)>0$ on a punctured neighborhood of $s_0$, and $A_v(W)=0$ only on a discrete, and hence finite, set. 
    \end{proof}		
    We prepare for the rest of the proof of  Proposition \ref{prop:order} with a local averaging estimate:
    \begin{lemma}\label{l:WlocAvg} 
    Suppose $z_0 \in \mathcal{G}(v)$ is a point of order $\eta$ with $(\psi, U)$ as in Definition \ref{def:order}, and $W(z) \simeq d(z)^{\beta}$ in a neighborhood of $z_0$. Suppose $(s_0,t_0)= s_0 v^{\bot} + t_0 v = z_0$ and let $L_s= L_{z_0+sv^{\bot},v}$. Assume there exists $\d>0$ such that  $L_s \cap \omega \cap U \neq \varnothing$ for all $s \in [s_0, s_0+\d)$, resp.\ $s \in (s_0-\d,s_0]$.    
    Then there exists $C \geq 1, \e>0$ such that for $s \in [s_0,s_0+\e),$ resp.\ $s \in (s_0-\e,s_0]$, we have
    \begin{enumerate}
        \item When $\eta  \geq 1$,
            \begin{align}
                \frac{1}{C} |s - s_0|^{\beta + 1/\eta} 
                \leq \int_{L_s \cap U} W(s,t) dt \leq 
                C |s - s_0|^{\beta + 1/\eta} .
            \end{align}
            
        \item When $\eta \leq 1$,
            \begin{align}
                \frac{1}{C} |s - s_0|^{(\beta + 1)/\eta} 
                \leq \int_{L_s \cap U} W(s,t) dt \leq 
                C |s - s_0|^{(\beta + 1)/\eta} .
            \end{align}
    \end{enumerate}
    \end{lemma}
    \begin{proof}
    We will consider the case where $L_s \cap \omega \cap U \neq \varnothing$ for all $s \in [s_0,s_0+\d)$ as the proof for $s \in (s_0-\d, s_0]$ is analogous. Without loss of generality we may also assume that $a>0$ in \eqref{eq:psicoords} so that $s>s_0$ is mapped to $x>0$ by $\psi$.
    
        
        When $\e>0$ is taken small enough, for $(s,t) \in \omega \cap U$ we have
        \begin{equation}
            C^{-1} \dist\left( (s,t), \omega^c \cap U\right)^{\beta}
                \le W(s,t)
                \le C \dist\left( (s,t), \omega^c \cap U\right)^{\beta }.
        \end{equation}
        By its definition $\psi$ is bi-Lipschitz. Using this, \eqref{eq:psicoords}, and  writing $\ti{t}=b(s-s_0)+c(t-t_0)$, 
        \begin{align}
            \dist((s,t),\omega^c \cap U) \simeq  \dist(\psi(s,t), \psi(\omega^c \cap U)) =\dist\left(\begin{pmatrix}a(s-s_0) \\ \ti{t} \end{pmatrix}, \psi(\omega^c \cap U) \right) . \label{eq:JbiLipschitz}
        \end{align}
        Now let 
        \begin{align}
            \Gamma_{out} &=\{(x,y) \in \psi(U) : C_{out} |y|^{\eta} \leq x\}^c , \\
            \Gamma_{in} &= \{(x,y) \in \psi(U) : y \geq 0,  C_{in}^{-1} |y|^{\eta} \leq x \leq C_{in}|y|^{\eta}\}^c.
        \end{align}
        By Definition \ref{def:order},
        \begin{equation}
            \Gamma_{in} \supset \psi(\omega^c \cap U) \supset \Gamma_{out}.
        \end{equation}
        Therefore 
        \begin{equation} 
            \dist(\psi(z), \Gamma_{in}) \leq \dist(\psi(z), \psi(\omega^c\cap U)) \leq \dist(\psi(z), \Gamma_{out}),
        \end{equation}
        and so
        \begin{equation} \label{e:sandwich-in-out}
            \dist(\psi(z), \Gamma_{in}) \lesssim \dist(z, \omega^c \cap U) \lesssim \dist(\psi(z), \Gamma_{out}).
        \end{equation}
        We now focus on the upper bound. By Lemma \ref{l:distancelemma-fang}, we have for any $z$ such that $\psi(z) \not \in \Gamma_{out}$:
        \begin{equation} \label{e:dist-Gamma-out}
        \begin{split}            \dist(\psi(z),\Gamma_{out})=\dist\left(\begin{pmatrix}
            a(s-s_0) \\ \ti{t}
        \end{pmatrix}, \Gamma_{out}\right)  \lesssim
            \left\{
            \begin{aligned}
            \left| |\tilde t|^{\eta} - C_{out}^{{-1}} a (s - s_0) \right|, & \;\; \eta \ge 1 , \\
            \left| |\tilde t| - (C_{out}^{{-1}} a (s - s_0))^{1/\eta} \right|, & \;\; \eta \le 1 .
            \end{aligned}
            \right.
        \end{split}
        \end{equation}
        When $\eta \geq 1$ we combine this with~\eqref{e:sandwich-in-out} to obtain
        \begin{align}
            \int_{L_s \cap U} W(s,t) dt~
                &{\lesssim} \int_{L_s \cap U} \dist\left( (s,t), \omega^c \cap U \right)^{\beta} d t \\
                &\lesssim \int_{|\tilde t|^{\eta} \le C_{out}^{{-1}} a (s - s_0)} \bigg(C_{out}^{{-1}} a(s-s_0) - |\ti{t}|^{\eta}\bigg)^{\beta} d\ti t \\
                &\lesssim |s - s_0|^{\beta + 1/\eta} ,
        \end{align}
        where we used Lemma~\ref{l:change-var} in the last step, with $(\eta, k, \epsilon) = (\eta, C_{out}^{{-1}} a (s - s_0), 1)$.
        This is exactly the upper bound in part 1.
        When $\eta \le 1$, we combine~\eqref{e:dist-Gamma-out} with~\eqref{e:sandwich-in-out} to obtain
        \begin{align}
            \int_{L_s \cap U} W(s,t) dt
                ~&{\lesssim} \int_{L_s \cap U} \dist\left((s,t), \omega^c \cap U \right)^{\beta} dt\\
                &\lesssim \int_{|\tilde t|^{\eta} \le C_{out}^{{-1}} a (s - s_0)} \bigg((C_{out}^{{-1}} a(s-s_0))^{1/\eta} - |\ti{t}|\bigg)^{\beta} d\ti t \\
                &\lesssim |s - s_0|^{(\beta + 1)/\eta} ,
        \end{align}
        where we used Lemma~\ref{l:change-var} in the last step, with $(\eta, k, \epsilon) = ( 1, (C_{out}^{{-1}} a (s - s_0))^{1/\eta}, 1)$.
        This is exactly the upper bound in part 2.

        Now we turn to the lower bound. By definition, $\psi(z) = (x, y) \not\in \Gamma_{in}$ if and only if
        \begin{equation*} 
         C_{in}^{-1} |y|^{\eta}\le x \le C_{in} |y|^{\eta} .
        \end{equation*}
        Therefore, by Lemma \ref{l:distancelemma-fang}, we have for any $z$ such that $\psi(z) \not\in \Gamma_{in}$:
        \begin{align}
        \dist(\psi(z), \Gamma_{in})&=\dist\left(\begin{pmatrix} a(s-s_0) \\ \ti{t} \end{pmatrix}, \Gamma_{in} \right) \\
        &\gtrsim
            \left\{
            \begin{aligned}
            \min\left\{ \left| |\tilde t|^{\eta} - C_{in} a (s - s_0) \right|, \left| |\tilde t|^{\eta} - C_{in}^{-1} a (s - s_0) \right| \right\} &\quad \textrm{if} \;\; \eta \ge 1 , \\
            \min\left\{ \left| |\tilde t| - (C_{in} a (s - s_0))^{1/\eta} \right|, \left| |\tilde t| - (C_{in}^{-1} a (s - s_0))^{1/\eta} \right| \right\} &\quad \textrm{if} \;\; \eta \le 1 .
            \end{aligned}
            \right.
        \end{align}
        Splitting into cases based on whether $|\ti{t}|^{\eta}$ is closer to $C_{in} a(s-s_0)$ or $C_{in}^{-1} a(s-s_0)$, with $\tilde C_{in} := \frac{C_{in} + C_{in}^{-1}}{2}$, we can rephrase this as follows: 
        if $\eta \ge 1$, we have
        \begin{align} \label{e:l-bound-g}
        \dist(\psi(z), \Gamma_{in})
            &\gtrsim \left( C_{in} a (s - s_0) - |\tilde t|^{\eta} \right) \mathbf{1}_{C_{in} a (s - s_0) \ge |\tilde t|^{\eta} \ge \tilde C_{in} a (s - s_0)} \nonumber\\
            &\qquad\qquad + \left( |\tilde t|^{\eta} - C_{in}^{-1} a (s - s_0) \right) \mathbf{1}_{C_{in}^{-1} a (s - s_0) \le |\tilde t|^{\eta} \le \tilde C_{in} a (s - s_0)} \nonumber\\
            &\ge \left( C_{in} a (s - s_0) - |\tilde t|^{\eta} \right) \mathbf{1}_{C_{in} a (s - s_0) \ge |\tilde t|^{\eta} \ge \tilde C_{in} a (s - s_0)} .
        \end{align}
        For $\eta \le 1$, we have similarly
        \begin{equation} \label{e:l-bound-l}
        \dist(\psi(z), \Gamma_{in}) \gtrsim
            \left( (C_{in} a (s - s_0))^{1/\eta} - |\tilde t| \right) \mathbf{1}_{C_{in} a (s - s_0) \ge |\tilde t|^{\eta} \ge \tilde C_{in} a (s - s_0)} .
        \end{equation}
        In the case where $\eta \ge 1$, we raise both sides of ~\eqref{e:l-bound-g} to the power $\beta$ and integrate in $\tilde t$, then we use the left-hand side of~\eqref{e:sandwich-in-out} to obtain
        \begin{equation*}
        \int_{L_s \cap U} \dist\left( (s,t), \omega^c \cap U \right)^\beta dt
            \gtrsim \int_{C_{in} a (s - s_0) \ge |\tilde t|^\eta \ge \tilde C_{in} a (s - s_0)} \left( C_{in} a (s - s_0) - |\tilde t|^\eta \right)^\beta \, d \tilde t .
        \end{equation*}
        Next, we decompose the right-hand side into two terms and apply Lemma~\ref{l:change-var} to each of them with $(\eta ,k) = (\eta, C_{in} a(s-s_0))$ and $\epsilon = 1$ or $\epsilon=\tilde C_{in}/C_{in} < 1$
        \begin{align*}
        \int_{L_s \cap U} W(s,t) dt~
            &{\gtrsim} \int_{L_s \cap U} \dist\left( (s,t), \omega^c \cap U \right)^{\beta} d t \\
            & \gtrsim \int_{|\tilde t|^{\eta} \le C_{in} a (s - s_0)} \left( C_{in} a (s - s_0) - |\tilde t|^{\eta} \right)^{\beta} \, d \tilde t \\
            &\hspace{1.5cm}    - \int_{|\tilde t|^{\eta} \le \tilde C_{in} a (s - s_0)} \left( C_{in} a (s - s_0) - |\tilde t|^{\eta} \right)^{\beta} \, d \tilde t \\
            &= c_0(1) \left(C_{in} a (s - s_0)\right)^{\beta + 1/\eta} - c_0\left(\tfrac{\tilde C_{in}}{C_{in}}  \right) \left(C_{in} a (s - s_0)\right)^{\beta + 1/\eta} \\
            &\gtrsim (s - s_0)^{\beta + 1/\eta}
        \end{align*}
        (recall that the function $c_0$ from Lemma~\ref{l:change-var} is increasing). Following the same process in the case $\eta \le 1$, this time choosing $\epsilon = 1$ or $\epsilon=\tilde C_{in}/C_{in} < 1$ and $(\eta ,k) $$= ( 1, (C_{in} a(s-s_0))^{1/\eta})$ in Lemma~\ref{l:change-var}, we arrive at
        \begin{equation*}
        \int_{L_s \cap U} W(s,t) dt
            ~{\gtrsim} \int_{L_s \cap U} \dist\left( (s,t), \omega^c \cap U \right)^{\beta} d t
            \gtrsim (s - s_0)^{(\beta + 1)/\eta} ,
        \end{equation*}
        which yields the desired lower bound.
    \end{proof}
Now we prove the averaging estimate \eqref{eq:ordergrowavg2} near intervals where the average is zero. 
    \begin{proof}[Proof of \eqref{eq:ordergrowavg2}] 
   Let $[\alpha, \rho]$ be one of the intervals  from Proposition \ref{prop:order}(1). Then  $A_v(W)(s)=0$ for $s \in [\alpha,\rho]$, and for some $\e>0$, $A_v(W)(s)>0$ for all $s \in (\alpha-\e,\alpha)\cup(\rho,\rho+\e)$. We will only estimate $A_v(W)$ for $s \in [\rho, \rho+\e)$ as an analogous argument applies for $s$ near $\alpha$.
    
    In view of Proposition~\ref{prop:order}(1) and the definition of $A_v(W)$, for all $t$ such that $(\rho,t) \in {\mathcal G}(v)$, we have $(\rho,t) \in {\mathcal G}_1(v)$. By Remark \ref{r:finiteG} we know ${\mathcal G}_1(v)$ is finite. Thus there are finitely many points $(\rho, t_k) \in {\mathcal G}_1(v)$. Each point $(\rho, t_k)$ is a point of order $\eta_k$, we let $U_k$ be the associated coordinate neighborhood from Definition \ref{def:order}. We can also assume that they are pairwise disjoint. We also have for some $\beta_k$ that $W(z) \simeq d(z)^{\beta_k}$ near $(\rho, t_k)$.

    Then, letting $z_0=(\rho,t_1)$ and $L_s=L_{z_0+sv^{\bot},v}$, we have
    \begin{equation}
        \bigcup_{k=1}^n U_k \supset L_{\rho} \cap \p \omega = \bigcup_{k=1}^n \{(\rho, t_k)\}.
    \end{equation}
    Furthermore, there exists $\e>0$, such that for all $s \in [\rho,\rho+\e)$,
    \begin{equation}
        \bigcup_{k=1}^n U_k \supset L_{s} \cap  \omega.
    \end{equation}

    Therefore, for $s \in [\rho, \rho+\e)$, we have 
    \begin{equation}
        A(W)(s) = \frac{1}{T_v} \int_{0}^{T_v} W(s,t) dt =  \frac{1}{T_v} \sum_{k=1}^n \int_{L_s \cap U_k} W(s,t) dt.
    \end{equation}
    We apply Lemma~\ref{l:WlocAvg} to each term of the sum to deduce that there exists $C \geq 1$, such that
    \begin{align}
        C^{-1} (s-\rho)^{\beta_v} \leq \int_{L_s \cap U_k} W(s,t) dt \leq C (s-\rho)^{\beta_v}, \qquad 
        \beta_v = \frac{\beta_k}{\min(\eta_k , 1)} + \frac{1}{\eta_k },
    \end{align}
    which yields \eqref{eq:ordergrowavg2}.
    \end{proof}

    Now we prove the  estimate \eqref{eq:ordergrowavg1} near isolated points where the average is zero. 
    \begin{proof}[Proof of \eqref{eq:ordergrowavg1}]
    After  some additional care to describe the geometry, the proof is similar to that of \eqref{eq:ordergrowavg2}. Consider one of the points $s_0$ from Proposition \ref{prop:order}(1). Thus $A_v(W)(s_0)=0$, and for some $\e>0$,
    \begin{equation}\label{e:avwsg0}
        A_v(W)(s)>0 \text{ for all }s \in (s_0-\e,s_0)\cup(s_0, s_0+\e).
    \end{equation}
    The glancing points on $L_{s_0}$ are the points $(s_0, t)$ for all $t$ such that $(s_0, t) \in \p \omega$. By Remark~\ref{r:finiteG},  $\Gc$ is finite. Thus, there are finitely many points $(s_0, t_k) \in \Gc(v) \cap L_{s_0}$, which may be one-sided or two-sided. Each $(s_0,t_k)$ is a point of some order $\eta_k$, and we let $U_k$ be the associated coordinate neighborhood from Definition \ref{def:order}. 
    We also have $W(z) \simeq d(z)^{\gamma_k}$ for some $\gamma_k$ on a neighborhood of $(s_0,t_k)$, either on one side or on both sides of $L_{s_0}$.

    
    
    As in the proof of  \eqref{eq:ordergrowavg2}, there exists $\e>0$, such that for all $s \in (s_0-\e, s_0+\e)$, we have $\bigcup_{k=1}^n U_k \supset L_{s} \cap \omega$. 
    Therefore, for $s \in (s_0-\e, s_0+\e)$ we have 
    \begin{equation}\label{eq:splitUj}
        A_v(W)(s) = \frac{1}{T_v} \int_0^{T_v} W(s,t) dt =  \frac{1}{T_v} \sum_{k=1}^n \int_{L_s \cap U_k} W(s,t) dt ,
    \end{equation}
    where $T_v$ is the length of $L_s$.

By Lemma \ref{l:WlocAvg}, for each $s$ and $k$ we have
\begin{equation}\label{e:wstless}
\int_{L_s \cap U_k} W(s,t) dt \le C |s-s_0|^{\gamma_v}, \qquad \gamma_v = \frac{\gamma_k}{\min(\eta_k,1)} + \frac 1 {\eta_k}.
\end{equation}
Furthermore, for each $s$,
\begin{equation}\label{e:wstgtr}
\int_{L_s \cap U_k} W(s,t) dt \ge C^{-1} |s-s_0|^{\gamma_v},
\end{equation}
for all $k$ such that the left side is not zero, and such a $k$ exists for each $s \ne s_0$ by \eqref{e:avwsg0}. Hence \eqref{eq:ordergrowavg1} follows from inserting \eqref{e:wstless} and \eqref{e:wstgtr} into \eqref{eq:splitUj}.        
    \end{proof}

\section{Proofs of results for polygonal damping sets} \label{sec:generic-polygon}

In this section we prove Proposition~\ref{p:polygon} and Corollary \ref{c:polygon}.

\begin{proof}[Proof of Proposition~\ref{p:polygon}]
First, according to Lemma~\ref{l:finite-trapped}, for any $r > 0$, there is a finite set $\Upsilon_r \subset \Sb^1$ of rational directions on $\Tb^2$ such that any geodesic with direction $v \not\in \Upsilon_r$ enters every open ball of radius $r$ in $\Tb^2$. Note that if  $r_1 \le r_2$, then $\Upsilon_{r_2} \subset \Upsilon_{r_1}$.

Let $\mathcal{Q}'$ be the set of polygons $P \in \mathcal{P}_n$ whose directions of edges do not belong to $\Upsilon_{1/k_P}$, where $k_P := \lfloor \frac{1}{r(\omega_P)} \rfloor + 1$, and 
\begin{equation} \label{e:inradius}
r = r(\omega_P)
	:= \sup \left\{ \epsilon > 0 \;:\; \exists z \in \Tb^2 : B_{\e}(z) \subset \omega_P \right\} .
\end{equation}
Let $\mathcal{Q}$ be the interior of $\mathcal{Q}'$.

Introduce for any rational direction $v \in \Sb^1$ and any $k \in \{1, 2, \dots, n\}$ the set:
\begin{equation} \label{e:cal-E}
\mathcal{E}_{v, k} := \Rb^{2(k-1)} \times \Span(v) \times \Rb^{2(n-k)}
    \subset \Rb^{2n},
\end{equation}
which is a subspace of dimension $2n - 1$. 
Note that for any $v$ and $k$, $\mathcal{H}_n \not \subset \mathcal{E}_{v,k}$, as $\mathcal{E}_{v,k}$ does not contain any polygon whose $k$th side is an irrational vector. Thus, $\mathcal{H}_n \cap \mathcal{E}_{v, k}$ is a proper subspace of $\mathcal{H}_n$, of dimension at most $\dim \mathcal{H}_n - 1 = 2n - 3$.

We will prove that $\mathcal Q$ is dense by showing that for any $P \in \mathcal P_n$ there is $\epsilon>0$ such that
\begin{equation}\label{e:calQ'}
    \mathcal{Q}' \cap B_{\epsilon}(P) \supset \Big(\mathcal{P}_n \backslash  \bigcup \{\mathcal{E}_{v, k} \cap \mathcal{H}_n \colon  1 \le k \le n , v \in \Upsilon_{1/k_P}\}\Big) \cap B_{\epsilon}(P);
\end{equation}
this suffices because the right hand side of \eqref{e:calQ'} is open and dense in $B_{\epsilon}(P)$. 

To prove \eqref{e:calQ'},  choose a representation of $P$, with vertices $x_1, x_2, \ldots, x_n$, and a point $x$ in the polygon at distance $r := r(\omega_P)$ from $\partial \omega_P$. For any $\epsilon > 0$ and any other polygon $\tilde P$ defined by a family of vertices $\tilde x_1, \tilde x_2, \ldots, \tilde x_n$ such that
\begin{equation*}
|\tilde x_j - x_j| \le \epsilon ,
	\qquad \forall j \in \{1, 2, \ldots, n\} ,
\end{equation*}
we have for all $j \in \{1, 2, \ldots, n\}$ and $t \in [0, 1]$:
\begin{align*}
|(1 - t) \tilde x_j + t \tilde x_{j+1} - x|
	&= |(1 - t) (\tilde x_j - x_j) + t (\tilde x_{j+1} - x_{j+1}) + (1 - t) x_j + t x_{j+1} - x| \\
	&\ge |(1 - t) x_j + t x_{j+1} - x| - (1 - t) |\tilde x_j - x_j| - t |\tilde x_{j+1} - x_{j+1}| \\
	&\ge r - \epsilon .
\end{align*}
We deduce that any perturbed polygon $\tilde P$ in $B_\epsilon(P)$ is such that $r(\omega_{\tilde P}) \ge r(\omega_P) - \epsilon$. Taking $\epsilon$ sufficiently small ensures that
\begin{equation*}
k_{\tilde P}
    := \left\lfloor \dfrac{1}{r(\omega_{\tilde P})} \right\rfloor + 1
    \le \left\lfloor \dfrac{1}{r(\omega_P) - \epsilon} \right\rfloor + 1 
    = \left\lfloor \dfrac{1}{r(\omega_P)} \right\rfloor + 1 
    =: k_P .
\end{equation*}
Therefore, we have $\Upsilon_{1/k_{\tilde P}} \subset \Upsilon_{1/k_P}$ for any $\tilde P \in B_\epsilon(P)$. By definition of $\mathcal{Q}'$, we deduce that
\begin{equation*}
\mathcal{Q}'^c \cap B_\epsilon(P)
    \subset \bigcup_{\substack{1 \le k \le n \\ v \in \Upsilon_{1/k_P}}} \mathcal{E}_{v, k} \cap \mathcal{H}_n \cap B_\epsilon(P) ,
        \qquad \dim \mathcal{E}_{v, k} \cap \mathcal{H}_n \le 2n - 3 ,
\end{equation*}
which implies \eqref{e:calQ'}.

We now assume that $P \in \mathcal{Q}$ and prove that the glancing set of $\omega_P$ consists of vertices of order~$1$. Let $L$ be a glancing line and $z \in L \cap \partial \omega_P$ and denote by $v$ the direction of $L$. We claim that $z$ is a vertex of the polygon. If not, since the polygon has no self-intersections, then either $v$ is tangent to the edge containing $z$, or $L$ enters the interior of $P$. In the first case $L \cap \partial \omega_P$ would contain an edge, which is impossible in view of the definition of $\mathcal{Q}' \supset \mathcal{Q}$ and Lemma~\ref{l:finite-trapped}. In the second case this contradicts $L$ being glancing at $z$. Thus our claim is true.

Now since the polygon has no self-intersections and projects properly onto the torus, we deduce that there are exactly two edges $v_j, v_{j+1}$ connected to this vertex, lying on the same side of $L$. We can construct an affine chart in a neighborhood of $z$ by mapping $(\delta (v_{j+1} - v_j), \delta v), \delta \ll 1$, to the standard basis vectors $(e_1, e_2)$ in $\Rb^2$.

Finally, we need to justify that for any $P \in \mathcal{P}_n$, only finitely many admissible angles $\theta \in \Theta(\omega_P)$ are such that $R_\theta P \not\in \mathcal{Q}$. This is true because by Lemma \ref{l:finite-trapped} there are only finitely many directions $v \in \Upsilon_{1/k_P}$ which can have glancing lines. For each $j \in \{1, 2, \ldots, n\}$, there are finitely many angles for which $R_\theta v_j$ coincides with one of these pathological directions, hence the result.
\end{proof}

We apply Theorem~\ref{thm:suff} to deduce Corollary~\ref{c:polygon}.

\begin{proof}[Proof of Corollary~\ref{c:polygon}]
Let $P \in \mathcal{Q}$. The glancing set consists of a finite number of points of order $1$ by Proposition~\ref{p:polygon}. In the worst-case scenario, we have ${\mathcal L}_1 \neq \varnothing$, hence by Theorem~\ref{thm:suff} decay at rate $\alpha = 1 - \frac{1}{\beta + 1 + 3}$ (otherwise the decay is faster). For $P \in \mathcal{P}_n$, this applies to all but a finite number of rotations of $P$ by Proposition~\ref{p:polygon}.
\end{proof}

\section{Proofs of results for $C^2$ damping sets} \label{sec:generic-curve}

In this section we prove Proposition~\ref{p:curve} and Corollary \ref{c:curve}.
	
\begin{proof}[Proof of Proposition~\ref{p:curve}]
We first construct the set $\mathcal{Y}$. On the one hand, for any $\gamma \in \mathcal{U}$, we have a number $r(\omega_\gamma) > 0$ (see~\eqref{e:inradius}), corresponding to the radius of a ball contained in $\omega_\gamma$, the open set enclosed by $\gamma$. On the other hand, according to Lemma~\ref{l:finite-trapped}, for any $r > 0$, there is a finite set $\Upsilon_r \subset \Sb^1$ of rational directions on $\Tb^2$ such that any geodesic with direction $v \not\in \Upsilon_r$ enters any open ball of radius $r$ in $\Tb^2$. We can also assume that whenever $r_1 \le r_2$, then $\Upsilon_{r_2} \subset \Upsilon_{r_1}$. 

For any $\gamma \in \mathcal{U}$, with the curvature $\kappa$ as in \eqref{e:def-curvature}, we define $f_\gamma \in C(\Sb^1; [0, + \infty))$ by
\begin{equation*}
f_\gamma(t)
    := \kappa\left(\gamma(t)\right) + \prod_{\pm} \prod_{v \in \Upsilon_{1/n}} \left| \dfrac{\dot{\gamma}(t)}{|\dot{\gamma}(t)|} \pm \dfrac{v}{|v|} \right| ,
        \quad n = \left\lfloor \frac{1}{r(\omega_\gamma)} \right\rfloor + 1 ,
        \qquad t \in \Sb^1 .
\end{equation*}
Let $\mathcal{Y}$ be the subset of $\mathcal{U}$ of curves $\gamma$ such that $f_\gamma$ does not vanish. For such a curve $\gamma$, the product in the definition of $f_\gamma$ vanishes at any glancing point, since glancing directions necessarily belong to $\Upsilon_{1/n}$. Hence we have $\kappa > 0$ at every glancing point.

Let us prove that $\mathcal{Y}$ is open. Consider $\gamma \in \mathcal{Y}$ and $\tilde \gamma \in \mathcal{U}$ such that $\| \tilde \gamma - \gamma \|_{C^2} \le \epsilon$. Then
\begin{equation*}
r(\omega_{\tilde \gamma})
    \ge r(\omega_\gamma) - \epsilon.
\end{equation*}
This implies for $\epsilon$ small enough
\begin{equation*}
n_{\tilde \gamma}
    := \left\lfloor \dfrac{1}{r(\omega_{\tilde \gamma})} \right\rfloor + 1
    \le \left\lfloor \dfrac{1}{r(\omega_\gamma) - \epsilon} \right\rfloor + 1 
    = \left\lfloor \dfrac{1}{r(\omega_\gamma)} \right\rfloor + 1 
    =: n_\gamma.
\end{equation*}
Therefore, we have
\begin{equation} \label{e:ups-n}
\Upsilon_{1/n_{\tilde \gamma}} \subset \Upsilon_{1/n_\gamma} .
\end{equation}
Using the continuity of
\begin{equation*}
C^2(\Sb^1; \Rb^2) \ni \tilde \gamma
    \longmapsto  \left| \kappa\left(\tilde \gamma(t)\right) \right| + \prod_{\pm} \prod_{v \in \Upsilon_{1/n_\gamma}} \left| \dfrac{\dot{\tilde \gamma}(t)}{|\dot{\tilde \gamma}(t)|} \pm \dfrac{v}{|v|} \right|,
\end{equation*}
where we note that here we put $n_\gamma$ in place of $n_{\tilde \gamma}$, we deduce that the above function does not vanish for $\tilde \gamma$ sufficiently close to $\gamma$ in the $C^2$ topology. In view of~\eqref{e:ups-n}, we conclude that $f_{\tilde \gamma}$ does not vanish either, therefore $\mathcal{Y}$ is open.

Now, let $\gamma \in \mathcal{U}$ and let us show that ``most rotations" are in $\mathcal{Y}$. Consider the $C^1$ map
\begin{align*}
\Sb^1 \ni t &\longmapsto \dfrac{\dot \gamma(t)}{|\dot \gamma(t)|} \in \Sb^1.
\end{align*}
Its critical points are exactly the times at which $\kappa(\gamma(t)) = 0$, since the curvature is independent of the parametrization. By Sard's theorem, the set of critical values, that is to say the set $\mathcal{E} \subset \Sb^1$ of directions
\begin{equation*}
\frac{\dot \gamma(t)}{|\dot \gamma(t)|} \in \Sb^1
    \qquad \textrm{with} \qquad
\kappa\left(\gamma(t)\right) = 0
\end{equation*}
has Lebesgue measure $0$, and it is compact. 
Now note that $f_\gamma$ does not vanish for $\gamma$ rotated by some angle $\theta \in \Theta(\omega_\gamma)$, when $R_\theta \mathcal{E} \cap \Upsilon_{1/n_\gamma} = \varnothing$. Let us identify the directions in $\Upsilon_{1/n_\gamma}$ with angles in $\Sb^1$:
\begin{equation*}
\Upsilon_{1/n_\gamma}
    = \left\{ \theta_1, \theta_2, \ldots, \theta_k \right\} \subset \Sb^1 .
\end{equation*}
Then we have $R_\theta \mathcal{E} \cap \Upsilon_{1/n_\gamma} \neq \varnothing$ if and only if there exists $j \in \{1, 2, \ldots, k\}$ such that
\begin{equation*}
\theta_j \in R_\theta \mathcal{E}
    \quad \Longleftrightarrow \quad
\theta_j - \theta \in \mathcal{E}
    \quad \Longleftrightarrow \quad
- \theta \in R_{- \theta_j} \mathcal{E} .
\end{equation*}
We conclude that
\begin{equation*}
R_\theta \mathcal{E} \cap \Upsilon_{1/n_\gamma} = \varnothing ,
    \qquad \forall \theta \in \Sb^1 \setminus \bigcup_{j = 1}^k R_{-\theta_j}\mathcal{E} ,
\end{equation*}
and the above set of angles $\theta$ is the complement of a compact set of measure $0$.

We finally show that the set $\mathcal{Y}$ is dense in $\mathcal{U}$. For any $\gamma \in \mathcal{U}$, the set $\Theta(\omega_\gamma)$ is open, so it contains an open neighborhood of $0$, in which we can pick $\theta$, arbitrarily close to $0$, such that $R_\theta \gamma$ belongs to $\mathcal{Y}$.
\end{proof}

We apply Theorem~\ref{thm:suff} to deduce Corollary~\ref{c:curve}.

\begin{proof}[Proof of Corollary~\ref{c:curve}]
Let $\gamma \in \mathcal{Y}$. By Proposition~\ref{p:curve}, we know that the curvature is non-zero at every glancing point $z \in {\mathcal G}$. We can construct an affine chart in a neighborhood of $z$ by mapping $\delta (\ddot \gamma_1(0), \dot \gamma_1(0))$, $\delta \ll 1$, to $(e_1, e_2)$ the standard basis vectors in $\Rb^2$, where $\gamma_1$ is the arc-length reparameterization of $\gamma$ with $\gamma_1(0) = z$.  In this chart, we have $\gamma_1(t) = (t^2/2 \delta^2, t/\delta) + o(t^2)$, namely the point $z$ is of order $\eta = 2$. Theorem~\ref{thm:suff} yields decay at rate at least $1 - \frac{1}{\beta + 1/2 + 3}$.

Given $\gamma \in \mathcal{U}$, the above applies to rotations outside the compact negligible set $K$ from Proposition~\ref{p:curve}, which concludes the proof.
\end{proof}

    \appendix
	\section{Appendix} \label{app}

The following lemma allows us to evaluate integrals related to damping averages.

\begin{lemma} \label{l:change-var}
Let $\beta, \eta > 0$. Then there exists an increasing function $c_0 : [0, 1] \to (0, + \infty)$ such that
\begin{equation*}
\forall k > 0, \forall \epsilon \in [0, 1], \qquad
    \int_{|t|^\eta \le \epsilon k} \left(k - |t|^\eta\right)^\beta \, dt
        = c_0(\epsilon) k^{\beta + 1/\eta} .
\end{equation*}
\end{lemma}

\begin{proof}
By symmetry, it is sufficient to study the integral on $t \ge 0$. We introduce the change of variables $t' = k^{-1} t^\eta$, 
$dt = \eta^{-1} k^{1/\eta} {t'}^{1/\eta-1}\,dt'$,
so that
\begin{equation*}
\int_0^{(\epsilon k)^{1/\eta}} \left(k - t^\eta\right)^\beta \, dt
    =\int_0^{\epsilon} (k - kt')^\beta \eta^{-1} k^{1/\eta} {t'}^{1/\eta-1} \, dt'
    = k^{\beta + 1/\eta} \underbrace{\eta^{-1} \int_0^{\epsilon} (1 - t')^\beta t'^{1/\eta - 1} \, dt'}_{=: c_0(\epsilon)/2} ,
\end{equation*}
hence the result.
\end{proof}

	

Next, we prove a lemma stating that the distance between the curve $x=|y|^\eta$ and a point $(x_0,y_0)$ is approximated by the distance between $(x_0,y_0)$ and $(y_0^\eta, y_0)$. 
	
	\begin{lemma}\label{l:distancelemma-fang}
		For $\eta > 0$ and $c > 0$, let 
		\begin{equation}
			\Gamma_c = \{(x,y) \in \Rb^2:  c |y|^\eta=x\}.
		\end{equation}
		Consider $(x_0, y_0) \in \Rb^2$ with $x_0 \geq 0$. 
		\begin{enumerate}
			\item If $ \eta \geq 1$, there exist $C>1,\e>0$ such that if $x_0, |y_0|<\e$, then
			\begin{equation}
				C^{-1} \left||y_0|^{\eta}-c^{-1} x_0\right| \leq \dist\left( (x_0,y_0), \Gamma_c  \right) \leq C \left| |y_0|^{\eta}- c^{-1} x_0\right|.
			\end{equation}
			\item If $\eta \in (0,1]$, there exist $C >1, \e>0$ such that if $x_0, |y_0|<\e$, then 
			\begin{equation}
				C^{-1} \left| |y_0|- (c^{-1} x_0)^{1/\eta} \right| \leq \dist\left((x_0, y_0), \Gamma_c\right) \leq C \left||y_0|- (c^{-1} x_0)^{1/\eta}\right|.
			\end{equation}
		\end{enumerate}
	\end{lemma}
	\begin{figure}[h]
		\centering
		\begin{tikzpicture}
			[
			declare function={
				a=1.5;
			}]
			\draw[domain=0:a,samples=100,variable=\x,thick] plot ({\x},{\x^2});
			\draw[domain=0:a,samples=100,variable=\x,thick] plot ({\x},{-\x^2});
			\draw[->, thick] (0,-2*a)--(0,2*a);
			\draw[->, thick] (0,0)--(a*a,0);
			\filldraw[black] (1.2,.5) circle (2pt) node[anchor=west]{$z_0$};
			\filldraw[black] (.9,.81) circle (2pt);
			\node[anchor=east] at (.707,1) {$z_1$}; 
			\filldraw[black] (1.2,1.44) circle (2pt) node[anchor=west]{$z_0'$};
		\end{tikzpicture}
		\caption{The curve $x=|y|^\eta$ with $\eta<1$}
	\end{figure}
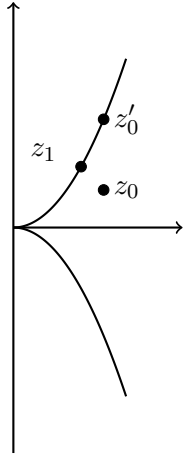

\begin{proof}
\emph{Case $\eta < 1$.}
Let $z_0 = (x_0, y_0) \in [0, + \infty) \times [0, + \infty)$ (by symmetry we can consider $y_0 \geq 0$). Let $z_1 = (x_1, y_1)$ be the point on $x=c|y|^\eta$ closest to $(x_0, y_0)$ and denote by $z_0' = (x_0, (c^{-1} x_0)^{1/\eta})$ the vertical projection of $z_0$ onto $\Gamma_c$.
In the triangle with vertices $z_0, z_0', z_1$, let us introduce the sidelengths
\begin{equation*}
d := |z_1 - z_0|,
	\qquad
d' := |z_0' - z_0|,
	\qquad \textrm{and} \qquad
\ell := |z_1 - z_0'| .
\end{equation*}
Our goal is to prove that $d' \sim d$ for $\e$ small enough.
The slope of the tangent to $\Gamma_c$ at a point $z = (x, (c^{-1}x)^{\frac{1}{\eta}})$ is equal to $\frac{c^{-1/\eta}}{\eta} x^{\frac{1}{\eta}-1}$. Note that as $z_0 \ra 0$, we have $z_1 \ra 0$ and $z_0' \ra 0$. Therefore, since $\eta<1$, the slope at $z_1$ and $z_0'$ tends to zero as $z_0 \to 0$, and so does 
the slope of the line spanned by $z_1 - z_0'$.
Hence,
\begin{equation} \label{e:slope1}
(z_1 - z_0') \cdot e_2
	= o(|z_1 - z_0'|)
	= o(\ell) .
\end{equation}
Since $z_1$ minimizes the distance from $z_0$ to $\Gamma_c$, $z_1-z_0$ is perpendicular to the tangent line of $\Gamma_c$ at $z_1$. Thus the slope of the line spanned by $z_1 - z_0$ tends to infinity as $z_0 \to 0$. Therefore we have
\begin{equation} \label{e:slope2}
\left|(z_1 - z_0) \cdot e_2\right|
	= |z_1 - z_0| + o(|z_1 - z_0|)
	= d + o(d) .
\end{equation}
Multiplying~\eqref{e:slope1} and~\eqref{e:slope2} by $d' = |z_0 - z_0'|$, since $z_0-z_0'$ is parallel to $e_2$, we deduce that
\begin{equation*}
\left\{
\begin{aligned}
\left|(z_1 - z_0') \cdot (z_0' - z_0) \right|
	&= o(\ell d') , \\
\left|(z_1 - z_0) \cdot (z_0 - z_0')\right|
	&= d d' + o(d d') .
\end{aligned}
\right.
\end{equation*}
Notice that, in the second equation, we have $(z_1 - z_0) \cdot (z_0 - z_0') \le 0$, since $y_1 - y_0$ and $y_0 - y_0'$ always have opposite signs, regardless of whether $z_0$ is to the right or to the left of $\Gamma_c$. Computing directly and applying the above equations we have
\begin{equation*}
\left\{
\begin{aligned}
d^2
	&= |z_1 - z_0|^2
	= |z_1 - z_0'|^2 + |z_0' - z_0|^2 + 2 (z_1 - z_0') \cdot (z_0' - z_0)
	= \ell^2 + d'^2 + o(\ell d') , \\
\ell^2
	&= |z_1 - z_0'|^2
	= |z_1 - z_0|^2 + |z_0 - z_0'|^2 + 2 (z_1 - z_0) \cdot (z_0 - z_0')
	= d^2 + d'^2 - 2 d d' + o(d d') .
\end{aligned}
\right.
\end{equation*}
Rearranging the first equation and then combining it with the second, we obtain
\begin{equation*}
d^2 - d'^2
	= \ell^2 + o(\ell d')
	= d^2 + d'^2 - 2dd' + o\left( \ell d' + d d' \right).
\end{equation*}
Rearranging again
\begin{equation*}
o\left( (\ell + d) d' \right)
	= d^2 - d'^2 - d^2 -d'^2 +2dd'
	= (d - d') 2 d' .
\end{equation*}
If $d' = 0$, the situation is trivial since then $z_0 = z_0' = z_1 \in \Gamma_c$. Otherwise, we have
\begin{equation*}
d - d'
	= o(\ell + d) .
\end{equation*}
Since $\ell \le d + d' \le 2 \max\{d, d'\}$, we conclude that $d \sim d'$ as $z_0 \ra 0$, which is the desired result for $\eta \in (0, 1)$.

\medskip

\emph{Case $\eta > 1$.}
For $\eta \in (1, + \infty)$, once again, we can assume without loss of generality that $y_0 \ge 0$ and that the distance $\dist(z_0, \Gamma_c)$ is achieved at a point $z_1 \in \Gamma_c \cap \{y \ge 0\}$, by symmetry with respect to $\{y = 0\}$. We reduce to the previous case by applying the isometry $T : (x, y) \mapsto (y, x)$ on $\Rb^2$: we have $T(\Gamma_c \cap \{y \ge 0\}) = \{x = c' |y|^{\eta'}\} \cap \{y \ge 0\}$ with $\eta' := 1/\eta < 1$ and $c'=c^{-1/\eta}$. Applying the previous case to $T(z_0) = (y_0, x_0)$ in place of $(x_0, y_0)$, we obtain
\begin{equation*}
\dist(z_0, \Gamma_c)
	= \dist\left(T(z_0), T(\Gamma_c)\right)
	\asymp \left| x_0 - \left(\frac{1}{c'} y_0\right)^{1/\eta'} \right|
	= \left| x_0 - c y_0^\eta \right| ,
\end{equation*}
dividing by $c$ proves the sought result for $\eta \in (1, + \infty)$.

\medskip

\emph{Case $\eta = 1$.}
For $\eta = 1$, we can compute the distance explicitly, for any $c > 0$: assuming $y_0 \ge 0$ again, the distance is achieved at $z_1 = (x_1, y_1) = (c y_1, y_1)$ such that
\begin{equation*}
0
	= \begin{pmatrix} x_0 - c y_1 \\ y_0 - y_1 \end{pmatrix} \cdot \begin{pmatrix} c \\ 1 \end{pmatrix}
	= y_1 \left( c x_0 + y_0 - (1 + c^2) y_1 \right) .
\end{equation*}
Therefore, by direct calculation,
\begin{equation*}
\dist(z_0, \Gamma_c)
	= \dfrac{|c y_0 - x_0|}{\sqrt{1 + c^2}} , 
		\qquad \forall z_0 \in \{x \ge 0\} ,
\end{equation*}
which finishes the proof.
\end{proof}


We now prove that the number of geodesics trapped outside $\omega$ is finite.

\begin{proof}[Proof of Lemma~\ref{l:finite-trapped}]
Geodesics with irrational direction are dense, so they intersect $\omega$. Let $v \in \Sb^1$ be a rational direction. Then $v = (p, q)/T_v$ for some coprime integers $p$ and $q$, where $T_v = (p^2 + q^2)^{1/2}$ is the period of any geodesic $L_{z, v}$, $z \in \Tb^2$. Let $a, b$ be B\'ezout coefficients, such that $ap + bq = 1$. Then we have
\begin{align*}
\dfrac{1}{T_v^2} \begin{pmatrix} -q \\ p \end{pmatrix}
	- \dfrac{bp - aq}{T_v^2} \begin{pmatrix} p \\ q \end{pmatrix}
		&= \begin{pmatrix} -b \\ a \end{pmatrix}
			\in \Zb^2 .
\end{align*}
We deduce that for any $z \in \Tb^2$, we have
\begin{equation} \label{e:transl-geod}
z + \dfrac{1}{T_v^2} \begin{pmatrix} -q \\ p \end{pmatrix}
	\equiv z + \dfrac{bp - aq}{T_v} v \pmod{\Zb^2}
	\in L_{z, v} .
\end{equation}
The vector $(-q, p)/T_v^2$ is orthogonal to $v$ and has norm $1/T_v$. If $1/T_v < 2 \varepsilon$, \eqref{e:transl-geod} implies that any geodesic of the form $L_{z, v}$ must intersect any ball of radius $\varepsilon$. Therefore, geodesics not entering $\omega$ must have length $T_v \le 1/2 \varepsilon$, hence $|p|, |q| \le 1/2 \varepsilon$, and the result follows.
\end{proof}

	\bibliographystyle{alpha}
	\bibliography{mybib}

\begin{thebibliography}{CSVW14}

\bibitem[AL14]{AL14}
N.~Anantharaman and M.~L{\'e}autaud.
\newblock Sharp polynomial decay rates for the damped wave equation on the
  torus.
\newblock {\em Anal. PDE}, 7(1):159--214, 2014.
\newblock With an appendix by S. Nonnenmacher.

\bibitem[BG20]{BurqGerard2018}
N.~Burq and P.~G{\'e}rard.
\newblock Stabilization of wave equations on the torus with rough dampings.
\newblock {\em Pure and Applied Analysis}, 2(3):627--658, 2020.

\bibitem[BH07]{BurqHitrik2007}
N.~Burq and M.~Hitrik.
\newblock Energy decay for damped wave equations on partially rectangular
  domains.
\newblock {\em Mathematical Research Letters}, 14(1):35--47, 2007.

\bibitem[BT10]{BorichevTomilov2010}
A.~Borichev and Y.~Tomilov.
\newblock Optimal polynomial decay of functions and operator semigroups.
\newblock {\em Mathematische Annalen}, 347(2):455--478, 2010.

\bibitem[Bur20]{Burq2020}
N.~Burq.
\newblock Decays for {K}elvin--{V}oigt damped wave equations {I}: The black box
  perturbative method.
\newblock {\em SIAM Journal on Control and Optimization}, 58(4):1893--1905,
  2020.

\bibitem[BZ12]{BurqZworski2012}
N.~Burq and M.~Zworski.
\newblock Control for schr{\"o}dinger operators on tori.
\newblock {\em Mathematical Research Letters}, 19(2):309--324, 2012.

\bibitem[BZ15]{BurqZuily2015}
N.~Burq and C.~Zuily.
\newblock Laplace eigenfunctions and damped wave equation on product manifolds.
\newblock {\em Applied Mathematics Research eXpress}, 2015(2):296--310, 2015.

\bibitem[BZ16]{BurqZuily2016}
N.~Burq and C.~Zuily.
\newblock Concentration of {L}aplace eigenfunctions and stabilization of weakly
  damped wave equation.
\newblock {\em Communications in Mathematical Physics}, 345:1055--1076, 2016.

\bibitem[BZ19]{BurqZworski2019}
N.~Burq and M.~Zworski.
\newblock {Rough controls for Schr\"odinger operators on tori}.
\newblock {\em Annales Henri Lebesgue}, 2:331--347, 2019.

\bibitem[CSVW14]{csvw}
H.~Christianson, E.~Schenck, A.~Vasy, and J.~Wunsch.
\newblock From resolvent estimates to damped waves.
\newblock {\em J. Anal. Math.}, 121(1):143--162, 2014.

\bibitem[DK20]{DatchevKleinhenz2020}
K.~Datchev and P.~Kleinhenz.
\newblock Sharp polynomial decay rates for the damped wave equation with
  {H}{\"o}lder-like damping.
\newblock {\em Proceedings of the American Mathematical Society},
  148(8):3417--3425, 2020.

\bibitem[Kle19]{Kleinhenz2019}
P.~Kleinhenz.
\newblock {Stabilization Rates for the Damped Wave Equation with
  H\"older-Regular Damping}.
\newblock {\em Commun. Math. Phys.}, 369(3):1187--1205, 2019.

\bibitem[Kle25]{Kleinhenz2025}
P.~Kleinhenz.
\newblock Sharp energy decay rates for the damped wave equation on the torus
  via non-polynomial derivative bound conditions.
\newblock {\em arXiv preprint arXiv:2502.09745}, 2025.

\bibitem[KW22]{KleinhenzWang2022}
P.~Kleinhenz and R.P.T. Wang.
\newblock Polynomially singular damping gives polynomial decay on the torus.
\newblock {\em arXiv preprint arXiv:2210.15697}, 2022.

\bibitem[LL17]{LeautaudLerner2017}
M.~L{\'e}autaud and N.~Lerner.
\newblock Energy decay for a locally undamped wave equation.
\newblock {\em Annales de la facult\'e des sciences de Toulouse S\'er.6},
  26(1):157--205, 2017.

\bibitem[LR05]{LiuRao2005}
Z.~Liu and B.~Rao.
\newblock Characterization of polynomial decay rate for the solution of linear
  evolution equation.
\newblock {\em Zeitschrift f{\"u}r angewandte Mathematik und Physik ZAMP},
  56(4):630--644, 2005.

\bibitem[{Mac}10]{Macia2010}
F.~{Maci\`a}.
\newblock {High-frequency propagation for the Schr\"odinger equation on the
  torus.}
\newblock {\em {J. Funct. Anal.}}, 258(3):933--955, 2010.

\bibitem[RT75]{RauchTaylor1975}
J.~Rauch and M.~Taylor.
\newblock Exponential decay of solutions to hyperbolic equations in bounded
  domains.
\newblock {\em Indiana Univ. Math. J.}, 24(1):79--86, 1975.

\bibitem[Sun23]{Sun23}
C.~Sun.
\newblock Sharp decay rate for the damped wave equation with convex-shaped
  damping.
\newblock {\em International Mathematics Research Notices}, 2023(7):5905--5973,
  2023.

\end{thebibliography}

\end{document}